\documentclass[preprint]{elsarticle}

\usepackage{hyperref}

\usepackage{amscd}
\usepackage{amsfonts}
\usepackage{amsmath}
\usepackage{amssymb}
\usepackage{amsthm}
\usepackage[mathscr]{euscript}
\usepackage{ifthen}
\usepackage{nameref}
\usepackage{rotating}
\usepackage{tikz}
\usetikzlibrary{matrix, arrows}
\usepackage{url}
\usepackage{wrapfig}
\usetikzlibrary{arrows}
\usetikzlibrary{decorations.pathmorphing}
\usetikzlibrary{decorations.markings}
\usetikzlibrary{shapes}
\usetikzlibrary{backgrounds}
\usetikzlibrary{patterns}

\theoremstyle{plain}
\newtheorem{theorem}{Theorem}[section]
\newtheorem{lemma}[theorem]{Lemma}
\newtheorem{proposition}[theorem]{Proposition}
\newtheorem{corollary}[theorem]{Corollary}
\theoremstyle{definition}
\newtheorem{definition}[theorem]{Definition}
\theoremstyle{remark}
\newtheorem*{remark}{Remark}
\theoremstyle{plain}

\newcommand{\Aut}{\ensuremath{\mathrm{Aut}}}
\renewcommand{\bar}[1]{\ensuremath{\overline{#1}}}
\newcommand{\C}{\mathcal{C}}
\newcommand{\D}{\mathcal{D}}
\newcommand{\E}{\mathcal{E}}
\newcommand{\eset}{\emptyset}

\newcommand{\FinSet}{\mathbf{FinSet}}
\newcommand{\G}{\mathcal{G}}

\renewcommand{\hat}[1]{\ensuremath{\widehat{#1}}}

\DeclareMathOperator*{\hocolim}{hocolim}

\newcommand{\Hom}{\ensuremath{\mathrm{Hom}}}

\newcommand{\initial}{\varnothing}

\newcommand{\ob}{\mathrm{ob}\,}
\newcommand{\op}{\mathrm{op}}

\renewcommand{\S}{\mathbb{S}}

\newcommand{\Sets}{\mathbf{Sets}}
\newcommand{\Sp}{\mathbf{Sp}}

\renewcommand{\tilde}[1]{\ensuremath{\widetilde{#1}}}

\AtBeginDocument{\DeclareMathSymbol{:}{\mathpunct}{operators}{"3A}}

\newcommand{\caseif}{&\mathrm{if}\ }
\newcommand{\caseotherwise}{&\mathrm{otherwise}}

\newcommand{\red}[1]{\textcolor[rgb]{1,0,0}{#1}}
\newcommand{\blue}[1]{\textcolor[rgb]{0,0,1}{#1}}
\newcommand{\green}[1]{\textcolor[rgb]{0,.5,0}{#1}}

\makeatletter
\def\myVCENTER#1{\vcenter{\hbox{$\m@th#1$}}}
\makeatother
\def\mathCHOICEtyped#1#2#3#4#5{#1{\mathchoice
{\myVCENTER{#2}}%
{\myVCENTER{#3}}%
{\myVCENTER{#4}}%
{\myVCENTER{#5}}}}
\def\smallerVERSION#1#2#3{\mathCHOICEtyped#1%
{\scriptstyle #2}%
{\scriptstyle #2}%
{\scriptscriptstyle #2}%
{\scriptscriptstyle #3}}

\renewcommand{\hat}[1]{\ensuremath{\widehat{#1}}}

\newcommand{\ASB}{\mathbf{Asm}}    
\newcommand{\bs}{\smallsetminus}
\newcommand{\Cat}{\mathbf{Cat}}

\renewcommand{\S}{\mathbb{S}}
\newcommand{\SCob}[2]{\{\uppercase{#1}_{#2}\}_{#2\in\uppercase{#2}}}

\renewcommand{\tilde}[1]{\ensuremath{\widetilde{#1}}}

\newcommand{\V}{\mathscr{V}}
\newcommand{\I}{\mathcal{I}}

\DeclareMathOperator{\hocofib}{hocofib}

\newcommand{\R}{\mathbb{R}}
\newcommand{\gG}{\mathfrak{G}}
\newcommand{\sP}{\mathfrak{P}}

\newcommand{\F}{\mathcal{F}}
\newcommand{\A}{\mathbb{A}}

\newcommand{\W}{\mathcal{W}}

\newcommand{\defeq}{\stackrel{\mathrm{def}}=}

\newcommand{\ps}[1]{\mathbf{#1}}

\newtheorem{maintheorem}{Theorem}

\newtheorem*{theorem:K0}{Theorem~\ref{thm:K0}}
\newtheorem*{theorem:subassemb}{Theorem~\ref{thm:subassemb}}
\newtheorem*{theorem:cofiber}{Theorem~\ref{thm:cofiber}}
\newtheorem*{theorem:cofibercalc}{Theorem~\ref{thm:cofibercalc}}
\newtheorem*{theorem:scissors}{Theorem~\ref{thm:scissors}}

\tikzset{
  move/.style={>->},
  sub/.style={densely dashed,->},
  cover/.style={densely dashed,->>}
}

\newcommand{\arrow}{\diagArrow[m-]}

\renewcommand{\dot}{{{\raisebox{.2ex}{\scalebox{.3}{$\bullet$}}}}}

\def\morpair#1#2#3#4{\begin{tikzpicture}[baseline=(A.base)] \node[anchor=east]%
    (A) at (0,0) {$#1$}; \node[anchor=west] (B) at (2em,0) {$#2$};%
    \diagArrow{<-, bend left}{A}{B}!{#3} \diagArrow{<-, bend %
      right}{A}{B}!{#4}%
  \end{tikzpicture}}

\def\smorpair#1#2#3{\begin{tikzpicture}[baseline=(A.base)] \node (A) at (0,0) {$S$}; \node
    (B) at (3em,0) {$#1$}; \diagArrow{<-, bend left}{A}{B}^{#2} \diagArrow{<-,
      bend %
      right}{A}{B}_{#3}%
  \end{tikzpicture}}
\newcommand{\bsmorpair}[3]{\Big[\smorpair{#1}{#2}{#3}\Big]}

\newtheorem*{quillena}{Quillen's Theorem A}
\theoremstyle{remark}
\newtheorem{example}[theorem]{Example}
\newtheorem{non-example}[theorem]{Counterexample}
\newtheorem{observation}[theorem]{Observation}
\theoremstyle{definition}

\theoremstyle{plain}

\providecommand{\isdraft}{false}
\providecommand{\draftdef}[4][0]{%
    \ifthenelse{\equal{\isdraft}{true}}{%
		\expandafter\newcommand\csname #2\endcsname[#1]{#3}}{%
		\expandafter\newcommand\csname #2\endcsname[#1]{#4}}}

\tikzset{
 >=stealth, 
 ^/.style={auto},
 under/.style={auto,swap}, 
 _/.style={auto,swap},
 description/.style={fill=white,inner sep=2pt}, 
 !/.style={fill=white,inner sep=2pt},
 raise/.style={above=.4ex},
 implies/.style={double,double equal sign distance,-implies},
 overbackground/.style={-,line width=6pt,draw=white},
 nodefont/.style={font=\normalsize},
}
\tikzstyle{matrix of math nodes}=[%
  matrix of nodes,
  nodes={%
   execute at begin node=$\displaystyle,%
   execute at end node=$%
  }%
]

\newenvironment{general-diagram}[2]{\begin{center}\begin{tikzpicture}%
\matrix (m) [matrix of math nodes, row sep=#1,%
column sep=#2]}{\end{tikzpicture}\end{center}}
\newenvironment{general-inline-diagram}[2]{\begin{tikzpicture}[baseline]%
\matrix (m) [matrix of math nodes, row sep=#1,%
column sep=#2]}{\end{tikzpicture}}
\newenvironment{general-fixed-diagram}[2]{\begin{center}\begin{tikzpicture}%
\matrix (m) [matrix of math nodes, row sep=#1,%
column sep=#2, text height=2.2ex, text depth=0.7ex]}{\end{tikzpicture}\end{center}}

\newenvironment{inline-diagram}[1][1.8em]{\begin{general-inline-diagram}{1.8em}{#1}}{\end{general-inline-diagram}}
\newenvironment{diagram}[1][2.5em]{\begin{general-diagram}{2.5em}{#1}}{\end{general-diagram}}

\newenvironment{squisheddiagram}[1][2.25em]{\begin{general-diagram}{.9em}{#1}}{\end{general-diagram}}
\newenvironment{diagram-fixed}[1][2.5em]{\begin{general-fixed-diagram}{2.5em}{#1}}{\end{general-fixed-diagram}}

\makeatletter
\def\innaisnum{\@ifnextchar 1{m-}{%
			   \@ifnextchar 2{m-}{%
			   \@ifnextchar 3{m-}{%
			   \@ifnextchar 4{m-}{%
			   \@ifnextchar 5{m-}{%
			   \@ifnextchar 6{m-}{%
			   \@ifnextchar 7{m-}{%
			   \@ifnextchar 8{m-}{%
			   \@ifnextchar 9{m-}{%
			   \@ifnextchar 0{m-} }}}}}}}}}}
\makeatother
\def\diagName#1{%
  \futurelet\diagChar\diagNameDecide#1}
\def\matrixPrefix{m-}
\def\diagNameDecide{%
	\ifx\diagChar 1\matrixPrefix%
	\else\ifx\diagChar 2\matrixPrefix%
	\else\ifx\diagChar 3\matrixPrefix%
	\else\ifx\diagChar 4\matrixPrefix%
	\else\ifx\diagChar 5\matrixPrefix%
	\else\ifx\diagChar 6\matrixPrefix%
	\else\ifx\diagChar 7\matrixPrefix%
	\else\ifx\diagChar 8\matrixPrefix%
	\else\ifx\diagChar 9\matrixPrefix%
	\else\ifx\diagChar 0\matrixPrefix%
	\fi\fi\fi\fi\fi\fi\fi\fi\fi\fi}

\newcommand{\placement}[1]{\if#1^auto\else\if#1_under\else\if#1-description\else #1\fi\fi\fi}

\def\diagDrawArrow#1#2#3#4#5{\path[font=\scriptsize] (#2) edge[#1] node[#4] {$#5$} (#3);}
\def\diagDrawWEArrow#1#2#3#4#5{\path[font=\scriptsize] (#2) edge[#1] node[outer %
  sep=.4ex, #4] {$#5$} node[above=-.7ex,sloped] {$\sim$} (#3);}

\newcommand{\diagArrow}[4][]{\diagArrowGo{#1}{#2}{#3}{#4}}

\def\diagArrowGo#1#2#3#4{%
    \def\diagDrawArrowHelper{\diagDrawArrow{#2}{#1#3}{#1#4}}  %
            \futurelet\diagArrowLookedAtToken\diagArrowDecide}			
\def\diagWEArrow#1#2#3{%
    \def\diagDrawArrowHelper{\diagDrawWEArrow{#1}{m-#2}{m-#3}}  %
            \futurelet\diagArrowLookedAtToken\diagArrowDecide}			
\def\diagDrawArrowHelperOther#1#2#3{\diagDrawArrowHelper{#2}{#3}}
\def\diagDrawArrowHelperBlank{\diagDrawArrowHelper{}{}}
\def\diagArrowDecide{%
   \ifx\diagArrowLookedAtToken^%
       \let\next=\diagDrawArrowHelper %
   \else\ifx\diagArrowLookedAtToken_ %
	   \let\next=\diagDrawArrowHelper %
   \else\ifx\diagArrowLookedAtToken! %
       \let\next=\diagDrawArrowHelper %
   \else\ifx\diagArrowLookedAtToken+ %
       \let\next=\diagDrawArrowHelperOther %
   \else %
       \let\next=\diagDrawArrowHelperBlank %
   \fi\fi\fi\fi %
   \next}

\newcommand{\tk}{\tikz[baseline, font=\scriptsize]}
\def\inlineDrawArrow#1^#2{\;\tk\draw[#1](0,0.5ex)-- node [above] {$#2$} (\inlinearrowlength,0.5ex);\;}
\def\inlineArrow#1{\def\inlineDrawArrowHelper{\inlineDrawArrow{#1}}  %
            \futurelet\inlineArrowLookedAtToken\inlineArrowDecide}
\def\inlineArrowDecide{%
   \def\inlineDrawArrowHelperBlank{\inlineDrawArrowHelper^{}}
   \ifx\inlineArrowLookedAtToken^%
       \let\next=\inlineDrawArrowHelper %
   \else\ifx\inlineArrowLookedAtToken_ %
	   \let\next=\inlineDrawArrowHelper %
   \else\ifx\inlineArrowLookedAtToken- %
       \let\next=\inlineDrawArrowHelper %
   \else %
       \let\next=\inlineDrawArrowHelperBlank %
   \fi\fi\fi %
   \next}
\def\inlineDrawWEArrow#1^#2{\;\tk\draw[#1, label distance=.5ex](0,0.5ex) -- node[label=above:$#2$, above=-.6ex] {$\sim$} (\inlinearrowlength,0.5ex);\;}
\def\inlineWEArrow#1{\def\inlineDrawArrowHelper{\inlineDrawWEArrow{#1}}  %
            \futurelet\inlineArrowLookedAtToken\inlineArrowDecide}

\def\inlineDoubleDrawArrow#1#2^#3{\;\begin{tikzpicture}[baseline, font=\scriptsize]%
    \draw[#1](0,1.1ex)-- node [above] {$#3$} (\inlinearrowlength,1.1ex);%
    \draw[#2](0,0.4ex)-- node {} (\inlinearrowlength,0.4ex); \end{tikzpicture}%
    \;}
\def\inlineDoubleArrow#1#2{\def\inlineDrawArrowHelper{\inlineDoubleDrawArrow{#1}{#2}}  %
            \futurelet\inlineArrowLookedAtToken\inlineArrowDecide}
			
\def\subscriptDrawArrow#1{\;\tk\draw[#1](0,0.5ex)-- node {} (\subscriptarrowlength,0.5ex);\;}
\def\subscriptArrow#1{\def\subscriptDrawArrowHelper{\subscriptDrawArrow{#1}}  %
            \futurelet\subscriptArrowLookedAtToken\subscriptArrowDecide}
\def\subscriptArrowDecide{%
   \def\subscriptDrawArrowHelperBlank{\subscriptDrawArrowHelper^{}}
   \ifx\subscriptArrowLookedAtToken^%
       \let\next=\subscriptDrawArrowHelper %
   \else\ifx\subscriptArrowLookedAtToken_ %
	   \let\next=\subscriptDrawArrowHelper %
   \else\ifx\subscriptArrowLookedAtToken- %
       \let\next=\subscriptDrawArrowHelper %
   \else %
       \let\next=\subscriptDrawArrowHelperBlank %
   \fi\fi\fi %
   \next}
   
\newcommand{\makelong}[2][\longinlinearrowlength]{{\def\inlinearrowlength{#1}#2}}
\newcommand{\makeshort}[2][\shortinlinearrowlength]{{\def\inlinearrowlength{#1}#2}}
\newcommand{\ms}[2][\shortinlinearrowlength]{\makeshort[#1]{#2}}
\everydisplay{\def\inlinearrowlength{\longinlinearrowlength}}

\def\inlinearrowlength{0.5}
\def\longinlinearrowlength{.75}
\def\shortinlinearrowlength{0.5}
\def\subscriptarrowlength{0.4}

\draftdef{rto}{\rightarrow}{\inlineArrow{->}}
\draftdef{longrto}{\longrightarrow}{\makelong\rto}
\draftdef{lto}{\leftarrow}{\inlineArrow{<-}}
\draftdef{rcofib}{\hookrightarrow}{\inlineArrow{right hook->}}
\draftdef{lcofib}{\hookleftarrow}{\inlineArrow{<-left hook}}
\draftdef{rfib}{\twoheadrightarrow}{\inlineArrow{->>}}
\draftdef{lfib}{\twoheadleftarrow}{\inlineArrow{<<-}}
\draftdef{rwe}{\stackrel{\sim}{\rightarrow}}{\inlineWEArrow{->}}
\draftdef{lwe}{\stackrel{\sim}{\leftarrow}}{\inlineWEArrow{<-}}
\draftdef{lacyccofib}{\stackrel{\sim}{\hookleftarrow}}{\inlineWEArrow{<-left hook}}
\draftdef{racyccofib}{\stackrel{\sim}{\hookrightarrow}}{\inlineWEArrow{right hook->}}
\draftdef{lacycfib}{\stackrel{\sim}{\twoheadleftarrow}}{\inlineWEArrow{<<-}}
\draftdef{racycfib}{\stackrel{\sim}{\twoheadrightarrow}}{\inlineWEArrow{->>}}
\draftdef{rgoesto}{\mapsto}{\inlineArrow{|->}}
\draftdef{lgoesto}{\inlineArrow{<-|}}{\inlineArrow{<-|}}
\draftdef{longrgoesto}{\longmapsto}{\makelong\rgoesto}
\draftdef{Rto}{\Rightarrow}{\inlineArrow{implies, -implies}}
\draftdef{Lto}{LImplies}{\inlineArrow{implies, implies-}}
\draftdef{req}{===}{\inlineArrow{implies,-}}
\draftdef{subrto}{\rightarrow}{\subscriptArrow{->}}
\draftdef{rrto}{\rightrightarrows}{\inlineDoubleArrow{->}{->}}
\draftdef{rlto}{\rightleftarrows}{\inlineDoubleArrow{->}{<-}}
\draftdef{rsub}{\dashrightarrow}{\inlineArrow{densely dashed,->}}

\def\to{\diagArrow[m-]{->}}

\def\eq{\diagArrow[m-]{implies, -}}

\newcommand{\arrowsquare}[5][->]{\diagArrow[m-]{#1}{1-1.mid east}{1-2.mid west}^{#2} \diagArrow[m-]{#1}{1-2}{2-2}^{#4} \diagArrow[m-]{#1}{2-1.mid east}{2-2.mid west}^{#5} \diagArrow[m-]{#1}{1-1}{2-1}_{#3}}
\newcommand{\largearrowsquare}[4]{%
\foreach \x in {#1,...,#2} %
  \foreach \y in {#3,...,#4} %
  {%
    \ifthenelse{\equal{\x}{#2}}{}{\pgfmathparse{int(1+\x)}%
	    \to{\x-\y}{\pgfmathresult-\y};}%
	\ifthenelse{\equal{\y}{#4}}{}{\pgfmathparse{int(1+\y)}%
	    \to{\x-\y}{\x-\pgfmathresult};}%
  }%
}

\begin{document}

\begin{frontmatter}

\title{The $K$-theory of Assemblers}
\author{Inna Zakharevich}
\address{The University of Chicago \\ Department of Mathematics \\ 5734 S
  University Ave, Rm 208 \\ Chicago, IL 60637}

\begin{abstract}
  In this paper we introduce the notion of an assembler, which formally encodes
  ``cutting and pasting'' data.  An assembler has an associated $K$-theory
  spectrum, in which $\pi_0$ is the free abelian group of objects of the
  assembler modulo the cutting and pasting relations, and in which the higher
  homotopy groups encode further geometric invariants.  The goal of this paper
  is to prove structural theorems about this $K$-theory spectrum, including
  analogs of Quillen's localization and d\'evissage theorems.  We demonstrate
  the uses of these theorems by analyzing the assembler associated to the
  Grothendieck ring of varieties and the assembler associated to scissors
  congruence groups of polytopes.
\end{abstract}

\begin{keyword}
\MSC[2010] 19D99 \sep  18F25 \sep 18F30
\end{keyword}

\end{frontmatter}


\maketitle



\section{Introduction}

Scissors congruence groups appear in many different areas of mathematics, from
classical geometry of Euclidean, spherical and hyperbolic space
\cite{dupontsah82, sah79, dupont01, goncharov99} to the study of birational
geometry and motivic integration \cite{larsenlunts03, looijenga02,
  nicaisesebag11} to questions about definable sets and logic \cite{dries98,
  hrushovskikazhdan06}.  However, as the contexts in which these appear are so
different there is no general theory for the study of such groups, and very
little exploration of their peculiar blend of geometry and combinatorics.  In
this paper we propose a framework for the study of these groups by replacing
them with $K$-theory spaces.  This framework is optimized to work with the
combinatorial nature of cutting and pasting sets together and does not rely on
any further algebraic structure.  

Our approach constructs an object called an \textsl{assembler} (see
Definition~\ref{def:assembler}), which encodes the data necessary to perform
scissors congruence.  The basic object underlying an assembler is a small
Grothendieck site where all morphisms are monic; the topology then encodes
exactly which objects ``assemble'' to form other objects.  Two maps $A \rto C$
and $B \rto C$ are considered disjoint if $A\times_C B$ exists and is equal to
the initial object.  We prove the following theorem in Section~\ref{sec:KC}.
\begin{maintheorem} \label{thm:K0}
  There exists a functor $K:\ASB \rto \Sp$ from the category of assemblers to
  the category of spectra such that for any assembler $\C$, $\pi_0K(\C)$ is the
  free abelian group generated by objects of $\C$ modulo the relations
  \[[A] = \sum_{i\in I} [A_i] \qquad \text{for any finite disjoint covering
    family } \{\makeshort{A_i \rto A}\}_{i\in I}.\]
\end{maintheorem}
The higher homotopy groups of assemblers encode further geometric information
about the geometry of cutting and pasting.  They are not generally trivial; for
example, for the assembler whose underlying category is $\initial \rto *$ and
which has the trivial topology, the $K$-theory has the homotopy type of the
sphere spectrum.  In a sequel \cite{Z-ass-pi1} we will construct generators for
$\pi_1K(\C)$ and show that these are related to ``self-scissors-congruences'' of
objects.

In this paper, however, we concern ourselves with more structural questions.  It
turns out that assemblers fall into a sweet spot in the definition of algebraic
$K$-theory similar to the sweet spot found by Quillen \cite{quillen73}.  When
Waldhausen \cite{waldhausen83} developed a framework for the algebraic
$K$-theory of spaces he had to discard many of the advantages of Quillen's exact
categories; in particular, while he had an analog of Quillen's Localization
Theorem (\cite[Theorem 5]{quillen73}, \cite[Proposition 1.5.5]{waldhausen83}) he
did not have an analog of Quillen's D\'evissage \cite[Theorem 4]{quillen73}.
This makes computations and analysis using Waldhausen's approach much more
difficult than Quillen's.  The approach used in this paper, while much more
analogous to Waldhausen's combinatorial approach than Quillen's algebraic one,
also has both localization and d\'evissage theorems.
\begin{maintheorem}[D\'evissage] \label{thm:subassemb}
  Let $\C$ be an assembler and $\D$ a full subassembler.  If for every object
  $A\in\C$ there exists a finite disjoint covering family $\{D_i \rto A\}_{i\in
    I}$ such that $D_i \in \D$ for all $i\in I$ then the induced map $K(\D) \rto
  K(\C)$ is an equivalence of spectra.
\end{maintheorem}
A subcategory $\D$ of an assembler $\C$ is called a \textsl{subassembler} if it
is an assembler and the inclusion $\D \rto \C$ is a morphism of assemblers.
This theorem is proved in Section~\ref{sec:subassemb}.  For examples of
applications of this theorem, see Sections~\ref{sec:gr} and \ref{sec:polytope}.

We also have a localization theorem; unfortunately, the most general form of the
theorem requires passing to simplicial assemblers.  We define a simplicial
assembler to be a functor $\Delta^{op} \rto \ASB$; for a simplicial assembler
$\C_\dot$ we define
\[K(\C_\dot) = \hocolim_{[n]\in \Delta^{op}} K(\C_n).\] Any assembler $\C$ can
be considered a simplicial assembler $\C_\dot$ by taking the constant functor at
$\C$.  It turns out that inside the category of simplicial assemblers it is
always possible to construct a nice model for the cofiber on $K$-theory of a
morphism of assemblers.  
\begin{maintheorem}[Localization] \label{thm:cofiber}
  Let $g:\D_\dot\rto \C_\dot$ be a morphism of
  simplicial assemblers.  There exists a simplicial assembler $(\C/g)_\dot$ with
  a morphism of simplicial assemblers $\iota:\C_\dot \rto (\C/g)_\dot$ such
  that the sequence
  \[K(\D_\dot) \rto^{K(g)} K(\C_\dot) \rto^{K(\iota)} K((\C/g)_\dot)\] is a
  cofiber sequence.
\end{maintheorem}
This is proved in Section~\ref{sec:cofiber}.  This theorem gives a useful formal
definition, but suffers from the same problem as many bar constructions: the
homotopy type of a simplicial object can be difficult to identify.  Even when
$\C_\dot$ and $\D_\dot$ are constant simplicial assemblers, $(\C/g)_\dot$ will
not be; this generally makes using Theorem~\ref{thm:cofiber} difficult.
However, in Section~\ref{sec:proof2} we show that for certain types of morphisms
of assemblers the cofiber turns out to be surprisingly simple.  Let $\C\bs \D$
be the full subcategory of $\C$ containing the objects
$((\ob \C) \bs (\ob \D)) \cup \{\initial\}$; for its assembler structure, see
Definition~\ref{def:CbsD} For an object $A$ of $\C$, we say that $\C$
\textsl{has complements for $A$} if any morphism $A \rto B$ is in a finite
disjoint covering family of $B$.
\begin{maintheorem} \label{thm:cofibercalc}  Let $\D$ be a subassembler of $\C$
  such that $\D$ is a sieve in $\C$ and such that $\C$ has complements for all
  objects of $\D$.  Then
  \[K(\D) \rto K(\C) \rto K(\C\bs\D)\]
  is a cofiber sequence.
\end{maintheorem}

As an application of this theory, in Section~\ref{sec:polytope} we construct a
spectral sequence relating the classical scissors congruence groups to
Goodwillie's total scissors congruence groups \cite{goodwillie14} and McMullen's
polytope algebra \cite{mcmullen89}.  The differentials in this spectral sequence
measure the difference between these groups; in particular, Goodwillie's groups
split as sums of classical groups if and only if certain differentials in the
spectral sequences are zero.
\begin{maintheorem} \label{thm:scissors} Let $\G_n$ be the assembler whose objects
  are polytopes of dimension $n$ in $E^n$, and let $\gG_n$ be the assembler
  whose objects are polytopes of dimension at most $n$; then the $\pi_0K(\G_n)$
  are the scissors congruence groups of Dupont and Sah and the $\pi_0K(\gG_n)$
  are the scissors congruence groups of Goodwillie.  There exists a spectral
  sequence
  \[E^1_{p,q} = \pi_p K(\G_q) \Rto \pi_p(\gG_n) \qquad\hbox{for }q \leq n.\]
\end{maintheorem}
A discussion of the differentials of this spectral sequence is given in
\cite[Section 5]{Z-ass-pi1}.

The theory developed in this paper will also be used in \cite{Z-ass-var} to prove the
following theorem:
\begin{maintheorem} \label{thm:kernelL}
  Let $K_0[\V_k]$ be the Grothendieck ring of varieties.  Any element $x$ in the
  kernel of multiplication by $[\A^1]$ can be represented as $[X] -[Y]$ where
  $X$ and $Y$ are varieties such that $[X\times \A^1] = [Y\times \A^1]$ but
  $X\times \A^1$ and $Y\times \A^1$ are not piecewise isomorphic.
\end{maintheorem}
Some preliminary work including a computation of a spectral sequence similar to
the one mentioned in Theorem~\ref{thm:scissors} is proved in Section~\ref{sec:gr}.

The theory of assemblers grew out of the theory of polytope complexes, developed
in \cite{zakharevich10, zakharevich11}.  Every polytope complex produces an
assembler (see Example~\ref{ex:poly-cmx}) but not vice versa, and assemblers are
much simpler to work with and much more flexible than polytope complexes.  For
this paper familiarity with polytope complexes is unnecessary, although some of
the approaches used here are similar to those used in \cite{zakharevich11}.

This paper is organized as follows.  In Section~\ref{sec:defs} we define
assemblers and the $K$-theory functor and prove Theorem~\ref{thm:K0}.  In
Section~\ref{sec:examples} we discuss several examples of assemblers and their
$K$-theories.  Section~\ref{sec:equiv} proves Theorem~\ref{thm:subassemb} and a
reduction theorem.  Section~\ref{sec:moreex} discusses two applications of the
theorems from Section~\ref{sec:equiv}, and proves Theorem~\ref{thm:scissors}.
Section~\ref{sec:cofiber} proves Theorem~\ref{thm:cofiber}.  Section
\ref{sec:proof2} proves Theorem~\ref{thm:cofibercalc}.

\subsection{Acknowledgments} I would like to thank Daniel Grayson, Mike Hopkins,
Peter May and Chuck Weibel for many useful and interesting conversations and for
looking at early drafts of this paper.  I would also like to thank the two
anonymous referees whose comments greatly improved this paper.

Funding: This work was supported by the National Science Foundation MSRFP Grant;
the Institute for Advanced Study; and the University of Chicago.

\section{Abstract Scissors Congruence} \label{sec:defs}

In this section we introduce scissors congruence of abstract objects.  We want
to define a relation $A\simeq B$ on certain kinds of objects, where we say that
$A\simeq B$ if $A$ can be ``decomposed'' into ``disjoint'' pieces
$A_1,\ldots,A_n$, and $B$ can be ``decomposed'' into ``disjoint'' pieces
$B_1,\ldots,B_n$ such that $A_i\cong B_i$ (for some definition of $\cong$).  To
be able to define this rigorously we introduce assemblers, which 
categorically codify this information in a natural way.

\subsection{Assemblers}

In this section we define the notion of an ``assembler,'' which is the
fundamental object of study of this paper.  For a discussion of several examples
of assemblers, see Section~\ref{sec:examples}.  

\begin{definition}
  In any category with an initial object $\initial$, we say that two morphisms
  $f:A\rto C$ and $g:B\rto C$ are \textsl{disjoint} if the pullback $A\times_C
  B$ exists and is equal to $\initial$.  A family $\{f_i:A_i \rto A\}_{i\in I}$
  is a \textsl{disjoint family} if for $i\neq i'$ the morphisms $f_i$ and
  $f_{i'}$ are disjoint.
\end{definition}

Although initial objects do not need to be unique, as all constructions in this
paper depend only on the non-initial objects in an assembler we will assume that
initial objects are unique.  

\begin{definition} \label{def:sieve}
  Let $\C$ be any category.  A \textsl{sieve} in $\C$ is a full subcategory $\D$
  such that for all $A$ in $\C$, if there exists a morphism $A\rto B$ in $\C$ with
  $B\in \D$, then $A$ is in $\D$.  In other words, a full subcategory $\D$ is a sieve
  in $\C$ if it is closed under precomposition with morphisms in $\C$.
\end{definition}

Note that any sieve in $\C$ must be equal to its essential image; in other
words, if $\D$ is a sieve in $\C$ and $A \cong A'$ with $A'\in \D$ then $A\in
\D$.  Observe that if $\mathcal{S}$ is a sieve in the over category $\C/C$ and
$f:B \rto C$ is a morphism in $\C$, then the preimage of $\mathcal{S}$ under
$f\circ\cdot:\C/B \rto \C/C$ is also a sieve, generally called $f^*\mathcal{S}$.

\begin{definition}
  A \textsl{Grothendieck topology} on a category $\C$ is a collection $J(C)$ of
  sieves in $\C/C$ for all objects $C\in \C$.  These collections must satisfy
  the following axioms:
  \begin{itemize}
  \item[(T1)] If $\mathcal{S}$ is in $J(C)$ and $f:B \rto C$ is a morphism in
    $\C$ then $f^*\mathcal{S}$ is in $J(B)$.
  \item[(T2)] Let $\mathcal{S}$ be in $J(C)$ and $\mathcal{T}$ be any sieve in
    $\C/C$.  If for every object $f:B \rto C$ in $\mathcal{S}$ the sieve
    $f^*\mathcal{T}$ is in $J(B)$ then $\mathcal{T}$ is in $J(C)$.
  \item[(T3)] $\C/C$ is in $J(C)$ for all $C\in \C$.
  \end{itemize}
  Given a family of morphisms $\{f_i:A_i \rto A\}_{i\in I}$ in $\C$, we say that
  it is a \textsl{covering family} if the full subcategory of $\C/A$ containing
  the objects
  \[\{g:X\rto A\,|\, \exists\ i\in I,\ h:X \rto A_i\hbox{ s.t. } f_ih = g\}\]
  is in $J(A)$.

  A category $\C$ with a Grothendieck topology is called a \textsl{Grothendieck
    site}.
\end{definition}

\begin{definition} \label{def:assembler} An \textsl{assembler} $\C$ is a small
  Grothendieck site satisfying the following extra conditions:
  \begin{itemize}
  \item[(I)] $\C$ has an initial object $\initial$, and the empty family is a
    covering family of $\initial$.
  \item[(R)] For any $A$, any two finite disjoint covering families of $A$ have
    a common refinement which is itself a finite disjoint covering family.
  \item[(M)] All morphisms in $\C$ are monomorphisms.
  \end{itemize}

  A \textsl{large assembler} is a Grothendieck site satisfying axioms (I), (R)
  and (M).  An assembler is said to be \textsl{closed} if it has all pullbacks.
\end{definition}

\begin{remark}
  If a Grothendieck site is closed under pullbacks then (R) always holds.
\end{remark}

The following lemma is direct from the definition of an assembler but is
important enough that we wish to highlight it:

\begin{lemma}
  If $A$ is noninitial in $\C$ then $\C(A,\initial) = \eset$.
\end{lemma}

We can now rephrase the definition of scissors congruence in the following way:

\begin{definition}
  Two objects $A,B$ in an assembler $\C$ are \textsl{scissors congruent},
  written $A\simeq B$, if there exist finite disjoint covering families $\{A_i
  \rto A\}_{i=1}^n$ and $\{B_i \rto B\}_{i=1}^n$ such that $A_i$ is isomorphic
  to $B_i$ for $1 \leq i \leq n$.
\end{definition}

\begin{definition}
  Let $\C$, $\D$ be two assemblers.  A functor $F: \C\rto \D$ is a
  \textsl{morphism of assemblers} if it is continuous (in the sense of
  Grothendieck topologies) and preserves the initial object and disjointness.
  In other words, if $f:A \rto C$ and $g:B\rto C$ are disjoint in $\C$ then
  $F(f)$ and $F(g)$ are disjoint in $\D$.  We denote the category of assemblers
  and morphisms of assemblers by $\ASB$.  The subcategory of closed assemblers
  and pullback-preserving morphisms of assemblers is denoted by $c\ASB$.
\end{definition}

For convenience, we denote the full subcategory of noninitial objects by
$\C^\circ$.  We have the following:

\begin{lemma} \label{lem:cop}
  $\ASB$ and $c\ASB$ have arbitrary products and coproducts.  
\end{lemma}

\begin{proof}
  Let $X$ be any set, and $\{\C_x\}_{x\in X}$ an $X$-tuple of assemblers.  We
  write $\bigvee_{x\in X} \C_x$ for the assembler whose class of objects is
  $\{\initial\}\cup \coprod_{x\in X}\ob \C_x^\circ$ and whose morphisms between
  noninitial objects come from the $\C_x$.  A morphism of assemblers
  $\bigvee_{x\in X} \C_x \rto \D$ is then just a morphism $F_x:\C_x\rto \D$ for
  each $x\in X$, so $\bigvee_{x\in X}\C_x$ is the coproduct in $\ASB$.

  We write $\prod_{x\in X} \C_x$ for the assembler whose underlying category is
  $\prod \C_x$ and whose topology is the product topology where a family is a
  covering family exactly when its projection to every coordinate is.  Then a
  morphism of assemblers $\D\rto \prod_{x\in X}\C_x$ is just an $X$-tuple of
  morphisms $\D\rto \C_x$, so this is the categorical product in $\ASB$.

  When restricting to $c\ASB$ the same proof applies.
\end{proof}

We will need one other construction on assemblers.
\begin{definition} \label{def:CbsD} Let $\C$ be an assembler and $\D$ a sieve in
  $\C$.  As a category, we define $\C\bs\D$ to be the full subcategory of $\C$
  containing all objects not in $\D^\circ$.  Then $\C\bs\D$ inherits an
  assembler structure from $\C$, where a family $\{f_i:A_i \rto A\}_{i\in I}$ in
  $\C\bs\D$ is defined to be a covering family if there exists a family of
  morphisms $\{f_j: A_j \rto A\}_{j\in J}$ such that each $A_j$ is in $\D$ for
  all $j\in J$ and such that $\{f_i:A_i \rto A\}_{i\in I\cup J}$ is a covering
  family in $\C$.
\end{definition}
In other words, a family in $\C\bs\D$ is a covering family if it can be
completed to a covering family in $\C$ by morphisms whose domains are in $\D$.
There is a natural morphism of assemblers
\[c:\C \rto \C \bs \D\]
sending each object in $(\ob \C)\bs (\ob \D)$ to itself and all objects in $\D$
to $\initial$.  A morphism $A \rto B$ with $A\notin \D$ is sent to itself; a
morphism $A \rto B$ with $A\in \D$ but $B\notin \D$ is sent to the unique
morphism $\initial \rto B$.

The fundamental construction allowing us to to scissors congruence with
assemblers is the category $\W(\C)$, which has as its objects formal sums of
objects of $\C$, and as its morphisms ``gluings'' of finite disjoint covering
families.  As the morphisms of this category keep track of the different ways of
``pasting'' objects together, its $K$-theory should contain important
information about scissors congruence classes.

\begin{definition}
  Let $\C$ be an assembler.  We define the category $\W(\C)$ to have objects
  $\{A_i\}_{i\in I}$, where $I$ is a finite set and $A_i$ is a noninitial object
  of $\C$ for all $i\in I$.  A morphism $f:\SCob{a}{i} \rto \SCob{b}{j}$ in
  $\W(\C)$ is a map of sets $f:I \rto J$ together with a tuple of morphism
  $f_i:A_i \rto B_{f(i)}$ for all $i\in I$ such that for all $j\in J$ the family
  $\{f_i:A_i \rto B_j\}_{i\in f^{-1}(j)}$ is a finite disjoint covering family.
  Note that $\W$ is a functor $\ASB\rto \Cat$.
\end{definition}

\begin{proposition} \label{prop:Wprop}
  Let $\C$ be an assembler.  
  \begin{itemize}
  \item[(1)] All morphisms in $\W(\C)$ are monomorphisms.
  \item[(2)] Any diagram $A \rto C \lto B$ in $\W(\C)$ can be completed to a
    commutative square.  If $\C$ is closed then $\W(\C)$ has all pullbacks.
  \item[(3)] For any family of assemblers $\{\C_x\}_{x\in X}$, let $\bigoplus
    \W(\C_x)$ be the full subcategory of $\prod \W(\C_x)$ where all but finitely
    many of the objects are the object indexed by the empty set.  The functor
    \[P:\W(\bigvee_{x\in X} \C_x) \rto \prod_{x\in X} \W(\C_x)\] induced by the
    morphisms of assemblers $F_x: \bigvee_{x\in X} \C_x \rto \C_x$ which sends
    all $\C_y$ for $y\neq x$ to the initial object, and sends $\C_x$ to itself
    via the identity induces an equivalence of categories
    \[\W(\bigvee_{x\in X}\C_x) \rto \bigoplus_{x\in X}\W(\C_x).\]
  \end{itemize}
\end{proposition}

\begin{proof}
  We prove these in turn.
  \noindent
  Proof of (1): Let $f:\SCob{A}{i} \rto \SCob{B}{j}$ be a morphism of $\W(\C)$,
  and consider two morphisms $g,h: \SCob{C}{k} \rto \SCob{A}{i}$ such that $fg =
  fh$; we show that $g=h$.  Consider any $k\in K$.  We have $fg(k) = fh(k)$,
  which implies that the square
  \begin{diagram}[3em]
    { C_k & A_{g(k)} \\ A_{h(k)} & B_{fg(k)} \\};
    \arrowsquare{g_k}{h_k}{f_{g(k)}}{f_{h(k)}}
  \end{diagram}
  commutes.  If $g(k) \neq h(k)$ then $f_{g(k)}$ and $f_{h(k)}$ are disjoint,
  and thus $C_k = \initial$, a contradiction.  Thus $g(k) = h(k)$.  But as each
  morphism in $\C$ is a monomorphism, this means that we must have $g_k = h_k$
  as well.

  \noindent
  Proof of (2):
  This follows directly from axiom (R); the second part follows from the
  definition of $\W$ if $\C$ has pullbacks.

  \noindent
  Proof of (3): As each object of $\W(\bigvee_{x\in X}\C_x)$ is indexed by a
  finite set, $P$ is actually a functor $\W(\bigvee_{x\in X} \C_x) \rto
  \bigoplus_{x\in X}\W(\C_x)$.  To see that $P$ is an equivalence, note that it
  is full and faithful and hits all objects indexed by disjoint indexing sets.
  Essential surjectivity follows because for any finite tuple of finite sets
  there exists an isomorphic one where the sets are disjoint.
\end{proof}

\subsection{The $K$-theory of an assembler} \label{sec:KC}

We are now ready to define the $K$-theory of an assembler.  

Recall that the topological analog of the tensor product is the smash product
$\wedge$, defined for pointed simplicial sets $X$ and $Y$ by
\[(X\wedge Y)_n \stackrel{\text{def}}{=} X_n\times Y_n / \left((X_n\times \{*\})
  \cup (\{*\}\times Y_n)\right).\] Let $S^1$ be the pointed simplicial set
$\Delta^1/(\partial \Delta^1)$, whose set of $n$-simplices is
$\{*,1,\ldots,n\}$.  Let $S^k = (S^1)^{\wedge k}$.  We have an action of
$\Sigma_k$ on $S^k$ which permutes the $S^1$ factors.

We write $N:\Cat \rto s\Sets$ for the nerve of a category.

\begin{definition} \label{def:redK} For a pointed set $X$ and an assembler $\C$,
  write $X\wedge \C$ for the assembler $\bigvee_{x\in X\bs\{*\}} \C$.  This
  gives a tensoring of $\ASB$ over $\FinSet_*$ (for the definition of a
  tensoring, see for example \cite[Section 3.7]{kelly82}).  For any pointed set
  $X$ write $X^\circ = X \bs \{*\}$.  Then we have an induced map \[X\wedge
  N\W(\C) \rto N\W(X \wedge \C)\] given by
  \[X\wedge N\W(\C) \cong \bigvee_{X^\circ} N\W(\C) \rto 
  N\bigg(\bigoplus_{X^\circ} \W(\C)\bigg) \cong N\W(X\wedge \C).\]
  This map is natural in $X$.

  The spectrum $K(\C)$ is defined to be the symmetric spectrum of simplicial
  sets where the $k$-th space is given by the diagonal of the bisimplicial set
  \[[n] \rgoesto N\W((S^k)_n \wedge \C),\] with the $\Sigma_k$-action
  induced from the action on $S^k$.  Considering $S^1$ to be a bisimplicial set
  constant in one direction, we have a map of simplicial sets
  \begin{eqnarray*}
    && \varphi_n:(S^1 \wedge N\W(S^k \wedge \C))_n \cong (S^1)_n \wedge
    N\W((S^k)_n\wedge \C) \\ &&\qquad \rto N\W((S^1)_n \wedge (S^k)_n  \wedge \C) \cong
    N\W((S^{k+1})_n \wedge \C).
  \end{eqnarray*}
  The spectral structure map $S^1\wedge N\W(S^k\wedge \C) \rto
  N\W(S^{k+1}\wedge \C)$ is the diagonal of the map of
  bisimplicial sets $\varphi$.

  We write $K_i(\C)$ for $\pi_i K(\C)$.
\end{definition}

Theorem~\ref{thm:K0} now follows from the definition of $K$ and the following
theorem: 

\begin{theorem}
  For any assembler $\C$, $\pi_0K(\C)$ is the free abelian group generated by
  objects of $\C$ modulo the relations
  \[[A] = \sum_{i\in I} [A_i] \qquad \text{for any finite disjoint covering
    family } \{\makeshort{A_i \rto A}\}_{i\in I}.\]
\end{theorem}

\begin{proof}
  To find generators and relations on $K_0(\C)$ we will use the theory of
  $\Gamma$-spaces; for background on $\Gamma$-spaces, see for example
  \cite{segal74} or \cite{bousfieldfriedlander}.  To every $\Gamma$-space $X$
  there is associated a symmetric spectrum $\mathbf{B}X$, which has
  $(\mathbf{B}X)_n = |X(S^n)|$.  The functor
  $X:\ps{n} \rgoesto |N\W(\ps{n}\wedge \C)|$ is a special $\Gamma$-space by
  Proposition~\ref{prop:Wprop}(3), and the spectrum $\mathbf{B}X$ is exactly
  equal to $K(\C)$.  Thus to find $\pi_0K(\C)$ it suffices to find
  $\pi_0\mathbf{B}X$.

  Since $X$ is special, $\pi_0X(\ps{1})$ is a monoid with operation induced by
  \[\pi_0X(\ps 1)\times \pi_0X(\ps{1}) \cong \makeshort[4em]{\pi_0X(\ps{2}) \rto^{X(\ps{2}\rightarrow\ps{1})}
    \pi_0X(\ps{1})}.\]
  Here, the first bijection is induced by the functor $P$ from
  Proposition~\ref{prop:Wprop}(3).  Rewriting this in terms of $\W$ it is simply
  the functor
  \[\W(\C)\oplus \W(\C) \rto^{P^{-1}} \W(\C\vee\C) \rto^\mu \W(\C),\]
  where $P^{-1}$ is any inverse equivalence to $P$, for example the one taking a
  pair of objects $(\{A_i\}_{i\in I}, \{B_j\}_{j\in J})$ to the object
  $\{C_k\}_{k\in I\sqcup J}$, where $C_k= A_k$ in the left copy of $\C$ if $k\in
  I$ and $C_k = B_j$ in the right copy of $\C$ if $k\in J$; the functor $\mu$ is
  then just the functor induced by the fold map of assemblers.  This operation
  is therefore the operation which sums objects by taking the disjoint unions of
  their indexing sets.  By \cite[Section 4]{segal74} $\pi_0\mathbf{B}X$ is the group
  completion of this monoid, so $\pi_0\mathbf{B}X$ is a quotient of the free abelian
  group generated by the noninitial objects of $\C$.

  The relations on the group are induced by morphisms of $\W(\C)$.  A morphism
  $f:\SCob{A}{i} \rto \SCob{B}{j}$ can be written as $\coprod_{j\in J}
  \big(\{A_i\}_{i\in f^{-1}(j)} \rto \{B_j\}\big)$, which is a $J$-fold formal
  sum of the component morphisms.  Each component is a finite disjoint covering
  family, and gives the relation
  \[[B_j] =\sum_{i\in f^{-1}(j)} [A_i].\]
  In addition, any finite disjoint covering family $\{A_i\rto A\}_{i\in I}$
  gives a morphism $\SCob{A}{i} \rto \{A\}$.  Thus the relations on $\pi_0(X)$
  given by the morphisms exactly correspond to the finite disjoint covering
  families of noninitial objects in $\C$.

  It remains to check that the group in the statement of the theorem gives the
  same group as the description of $\pi_0K(\C)$ above.  The only difference
  between these two descriptions is the presence of the initial object.
  However, the initial object has an empty covering family (which is
  tautologically finite and disjoint), so by the relation in the statement of
  the theorem $[\initial] = 0$.  Thus its presence in the description does not
  affect the group.
\end{proof}

We introduce the notion of a simplicial assembler and its $K$-theory.

\begin{definition} \label{def:simpass} A \textsl{simplicial assembler} is a
  functor $\Delta^\op \rto \ASB$.  A \textsl{morphism of simplicial assemblers}
  is a natural transformation of functors.  We define the $K$-theory spectrum of
  a simplicial assembler $\C_\dot$ by
  \[K(\C_\dot) = \hocolim_{[n]\in\Delta^\op} K(\C_n).\]

  We write $s\ASB$ for the category of simplicial assemblers.
\end{definition}

\begin{remark}
  Recall that homotopy colimits of spectra can be computed levelwise.  A diagram
  $\Delta^\op \rto s\Sets$ is a bisimplicial set; the diagonal of the bisimplicial
  set is a model for the homotopy colimit of the diagram.  We can thus give an
  explicit model for $K(\C_\dot)$ in an analogous way to
  Definition~\ref{def:redK} as follows.

  For any simplicial set $X_\dot$ and any simplicial assembler $\C_\dot$ we
  define a simplicial assembler $X_\dot\wedge \C_\dot$ by $(X\wedge \C)_n = X_n
  \wedge \C_n$.  We define $K(\C_\dot)_k = \big| \W(S^k\wedge \C_\dot)\big|$.
  The spectral structure maps are constructed analogously to those in
  Definition~\ref{def:redK}.  In particular, if we consider an assembler to be a
  constant simplicial assembler, then $K(\C)_k = \big| \W(S^k\wedge \C)\big|$.
\end{remark}

\section{Examples} \label{sec:examples}

In this section we examine several examples of assemblers and their
$K$-theories.  

\begin{example}
  Let $*$ be the assembler whose underlying category is trivial.  This
  assembler has no noninitial objects, so $\W(*)$ is the trivial category, and
  thus $K(*) \cong *$.
\end{example}

\begin{example} \label{ex:sphere} Fix a discrete group $G$, and let $\S_G$ be the
  assembler with two objects, $\initial$ and $*$, one non-invertible morphism
  $\initial \rto *$ and with $\Aut\,(*) = G$.  Then $\W(\S_G)$ has as its
  objects the finite sets and its morphisms $I\rto J$ are isomorphisms $I\rto J$
  together with an element $g_i$ for all $i\in I$.  By the
  Barratt--Priddy--Quillen--Segal Theorem (see \cite{segal74}) $K(\S_G)$ is
  stably equivalent to $\Sigma_+^\infty BG$.  A homomorphism of groups $\varphi:
  H \rto G$ gives a morphism of assemblers $\S_H \rto \S_G$, which we also
  denote by $\varphi$.

  In the special case when $G$ is the trivial group, $K(\S_G)\simeq \S$; we
  write $\S$ instead of $\S_1$.
\end{example}

\begin{example}
  Let $\C$ be the partial order of all open subsets of a topological space $X$,
  with the usual Grothendieck topology.  Then $\C$ is an assembler.  However,
  the only way that we can have a finite disjoint covering family $\{U_i \rto
  U\}_{i\in I}$ is if the $U_i$ are disjoint connected components of $U$.  Thus
  for any connected open subset $U$ of $X$ there are no nontrivial finite
  disjoint covering families of $U$.  In this case $K(\C)$ is a wedge of
  spheres, one for each connected open subset of $X$, and is therefore
  completely uninteresting.
\end{example}

Here is an example relating assemblers to logic.

\begin{example} \label{ex:functors} For any small category $\mathcal{I}$ and any
  (large) assembler $\C$, let $[\mathcal{I},\C]$ be the category of functors
  $\mathcal{I} \rto \C$ and natural transformations between them.  We define a
  topology on $[\mathcal{I},\C]$ by saying that $\{F_\alpha \rto F\}_{\alpha\in
    A}$ is a covering family if for all objects $X$ in $\C$, the family
  $\{F_\alpha(X) \rto F(X)\}_{\alpha\in A}$ is a covering family.  Then
  $[\mathcal{I},\C]$ is a (large) assembler.

  As a special case, let $\C=\Sets_i$ be the large assembler of sets and
  injective functions, and let $\mathcal{I}$ be the category of models and
  elementary inclusions of a logical theory $T$ (for an introduction to model
  theory, see for example \cite{hodges-shorter}).  Since any subcategory of an
  assembler containing the initial object and satisfying (R) is also an
  assembler, we can conclude that the category of definable sets of a theory is
  also an assembler.  $K_0$ of this assembler is the abelian group underlying
  the Grothendieck ring of definable sets.
\end{example}

\begin{remark}
  Section~\ref{sec:gr} and Examples~\ref{ex:sphere} and \ref{ex:functors} mention
  Grothendieck rings but we have not yet introduced any structures on assemblers
  that can lead to ring structures on $K_0$.  In fact, it turns out that the
  category of assemblers is a symmetric monoidal category and $K$ is a symmetric
  monoidal functor, which allows us to construct $E_\infty$ ring structures on
  the $K$-theory spectra of these assemblers.  This will be discussed in more
  detail in future work.
\end{remark}

Using the fact that $K(\S) \simeq \S$, we can construct homotopy types of all
suspension spectra.
\begin{example}
  For any pointed simplicial set $X_\dot$, $X_\dot \wedge \S$ is a simplicial
  assembler.  By definition,
  \[K(X_\dot\wedge \S) \simeq \hocolim_n K(X_n\wedge \S) \simeq \hocolim_n
  \bigvee_{X_n\bs\{*\}} \S \simeq \Sigma^\infty X_\dot.\] Thus all homotopy
  types of suspension spectra are in the image of the functor $K:s\ASB \rto
  \Sp$.
\end{example}

The original theory of \cite{zakharevich10, zakharevich11} led to the theory of
assemblers; the following example shows that assemblers are a strict
generalization of the ideas in those papers.

\begin{example} \label{ex:poly-cmx} The definition of a polytope complex and the
  notation $\cdot \lto \cdot \rsub \cdot$ can be found in \cite[Definition
  3.1]{zakharevich10}.  Given a polytope complex $\C$ we define an assembler
  $\C_a$ in the following manner.  We set $\ob \C_a = \ob \C$, and a morphism $A
  \rto B$ in $\C_a$ is given by an object $A'$ and a diagram
  \[A \lto^\sigma A' \rsub^p B \in \C.\] The composite of $A \lto^\sigma A'
  \rsub^p B$ and $B \lto^\tau B' \rsub^q C$ is given by \[A \lto^{\sigma\tau'}
  \tau^*A' \rsub^{q\tau^*p} C.\]  A family $\{A_i \lto A_i' \rsub A\}_{i\in I}$
  is a covering family exactly when $\{A_i' \rsub A\}_{i\in I}$ is a
  covering family in $\C$.  These definitions make $\C_a$ into an assembler, and
  $K(\C_a)\simeq K(\C)$.
\end{example}

The name ``polytope complex'' was inspired by the ideas of classical scissors
congruence; the new theory of assemblers allows us to make new connections in
those applications as well.

\section{Equivalences between assemblers} \label{sec:equiv}

In this section we explore two types of morphisms between assemblers which
produce equivalences on the level of $K$-theory.  The two main results,
Theorem~\ref{thm:subassemb} and Theorem~\ref{thm:SEp} are useful for
identifying assemblers that come up in various contexts.  For concretes
applications of these results, see Sections~\ref{sec:gr} and \ref{sec:polytope}. 

We begin with a result which is the basis of all of our assembler
calculations.  It is the ``inductive step'' which allows us to show that if we
have an equivalence of $K$-theory spectra at level $0$, then we must have an
equivalence of $K$-theory spectra at all levels.  Although the proof is
straightforward from previous results, we present it in its entirety here as it
is the key point in many analyses of assemblers.

By an abuse of notation, for a $k$-simplicial category $\E_{\cdots}$ we will write
$|\E_{\cdots}|$ for the simplicial set whose $n$-simplices are $N_n\E_{n\cdots
  n}$.

\begin{lemma} \label{lem:inductive-step}
  Suppose that $p:\C_\dot \rto \D_\dot$ is a morphism of simplicial assemblers
  such that the induced map of simplicial sets
  \[\big|\W(p)\big|: \big|\W(\C_\dot)\big| \rto \big|\W(\D_\dot)\big|\]
  is a weak equivalence of simplicial sets.  Then $K(p)$ is an equivalence of
  spectra. 
\end{lemma}

\begin{proof}
  It suffices to show that $K(p)_k$ is an equivalence for all $k$.  Thus we want
  to show that $\big| \W(S^k \wedge \C_\dot) \big| \rto \big| \W(S^k \wedge
  \D_\dot)\big|$ is a weak equivalence.  Let $S^k \barwedge \C_\dot$ be the
  bisimplicial assembler whose $(n,m)$-th entry is $S^k_n \wedge \C_m$.
  Applying $\W$ pointwise and applying $|\cdot|$, we see that 
  \[\big| \W(S^k \wedge \C_\dot)\big| \cong \big| \W(S^k \barwedge
  \C_\dot)\big|.\] Thus it remains to show that $p$ induces a weak equivalence
  \[\big| \W(S^k \barwedge \C_\dot) \big| \rto \big| \W(S^k \barwedge
  \D_\dot)\big|.\]  It suffices to show that for all $n$, $\big|
  \W(S^k_n \wedge \C_\dot) \big| \rto \big|\W(S^k_n \wedge \D_\dot)\big|$ is a
  weak equivalence.  We have the following commutative diagram:
  \begin{diagram}[5em]
    {\W(S^k_n \wedge \C_\dot) & \W(S^k_n \wedge \D_\dot) \\
      \W(\C_\dot)^{n^k} & \W(\D_\dot)^{n^k} \\};
    \arrowsquare{\W(S^k \wedge p)}{\simeq}{\simeq}{p^{n^k}}
  \end{diagram}
  The vertical morphisms are level equivalences of simplicial categories by
  Proposition~\ref{prop:Wprop}(3), and the bottom morphism is an equivalence
  after applying $|\cdot|$, as it is just an $n^k$-fold product of $p$
  with itself.  By two-of-three we can therefore conclude that $\W(S^k\wedge p)$
  is a weak equivalence after geometric realization.
\end{proof}

\begin{corollary} \label{cor:equiv-ass}
  If $F: \C\rto \D$ is a morphism of assemblers which is an equivalence on the
  underlying categories, then $K(F)$ is a weak equivalence of spectra.
\end{corollary}

\begin{proof}
  By Lemma~\ref{lem:inductive-step} it suffices to check that $N\W(F)$ is
  a weak equivalence.  However, since $F$ is an equivalence of categories it is
  straightforward that $\W(F)$ is, as well; since equivalences of categories are
  homotopy equivalences on nerves, we are done.
\end{proof}

We use Quillen's Theorem A to show that morphisms become equivalences after
applying $|\cdot|$.

\begin{quillena} \label{thm:quillena} Suppose that $F:\C\rto \D$ is a functor
  between small categories such that $N(F/Y)$ is contractible for all objects
  $Y$ in $\D$.  Then $N F:N\C\rto N\D$ is a homotopy equivalence.
\end{quillena}

We say a preorder $\C$ is \textsl{cofiltered} if for any two objects $A,B$ in
$\C$ there exists a diagram
\[A \lto X \rto B\]
in $\C$.  Quillen's Theorem A has the following corollary:

\begin{lemma} \label{lem:cofiltpre}   Any cofiltered preorder is contractible.
\end{lemma}
This lemma is the dual of the specialization to preorders of \cite[Corollary 2,
Section 1]{quillen73}.

\subsection{Restriction to subcomplexes} \label{sec:subassemb}

Recall that a subcategory $\D$ of an assembler $\C$ is called a
\textsl{subassembler} if it is an assembler and the inclusion $\D \rto \C$ is a
morphism of assemblers.  In this section we will prove Theorem~\ref{thm:subassemb}.

\begin{theorem:subassemb}
  Let $\C$ be an assembler and $\D$ a full subassembler.  If for every object
  $A$ there exists a finite disjoint covering family $\{D_i \rto A\}_{i\in
    I}$ such that $D_i$ is in $\D$ for all $i\in I$ then the induced map $K(\D) \rto
  K(\C)$ is an equivalence of spectra.
\end{theorem:subassemb}

\begin{proof}
  We begin with a couple of observations.  Since $\D$ is a full subcategory of
  $\C$, $\W(\D)$ is a full subcategory of $\W(\C)$.  In addition, for every
  object $A$ in $\W(\C)$ there exists an object $B$ in $\W(\D)$ and a morphism $B
  \rto A$ in $\W(\C)$.  To see this, write $A = \SCob{A}{i}$, and choose finite
  disjoint covering families $\{B^{(i)}_j\rto A_i\}_{j\in J_i}$ with
  $B^{(i)}_j$ in $\D$ for all $i\in I$ and $j\in J_i$.  Then setting $B =
  \{B^{(i)}_j\}_{(i,j)\in \coprod_{i\in I}J_i}$ gives us the desired object, and
  the covering families define the morphism $B \rto A$.

  By Proposition~\ref{prop:Wprop}(2), for any two morphisms $A \rto Y$ and $B\rto
  Y$ there exists a commutative square
  \begin{diagram}
    { C & A \\ B & Y \\};
    \arrowsquare{}{}{}{}
  \end{diagram}
  in $\W(\C)$; thus $\W(\C)$ is cofiltered.

  By Lemma~\ref{lem:inductive-step} it suffices to show that $i:\W(\D) \rto
  \W(\C)$ is an equivalence after geometric realization.  By Quillen's Theorem A
  it suffices to show that the category $i/Y$ is contractible for all $Y$ in
  $\W(\C)$.  As $\W(\C)/Y$ is a preorder and $\W(\D)$ is a subcategory of
  $\W(\C)$, $i/Y$ is also a preorder.  Thus by Lemma~\ref{lem:cofiltpre} it
  suffices to show that it is cofiltered.  To find an object above $A \rto Y$
  and $B \rto Y$ in $i/Y$ we complete
  \[A \rto Y \lto B\] to a square in $\W(\C)$ with apex $C$, and then choose a
  morphism $C' \rto C$ with $C'$ in $\W(\D)$.  As $\W(\D)$ is a full subcategory,
  the morphisms $C' \rto A$ and $C'\rto B$ are morphisms in $\W(\D)$, and
  therefore $i/Y$ is cofiltered.
\end{proof}

As levelwise weak equivalences in simplicial spectra map to weak equivalences
under realization, we immediately have the following:

\begin{corollary}
  Let $\D_\dot \rto \C_\dot$ be a morphism of simplicial assemblers such that
  for each $n$, $\D_n \rto \C_n$ is an inclusion of a subassembler with
  sufficiently many covers.  Then the map $K(\D_\dot) \rto K(\C_\dot)$ is a weak
  equivalence of spectra.
\end{corollary}

\subsection{Epimorphic assemblers with sinks}

\begin{definition}
  Let $\C$ be an assembler.  We say that $\C$ is an \textsl{epimorphic assembler
    with a sink} if it satisfies the following three conditions:
  \begin{itemize}
  \item[(S)] $\C$ contains an object $S$, called a \textsl{sink}, such that for
    all other objects $A$, $\C(A,S) \neq \eset$.
  \item[(Ep)] All morphisms with noninitial domain in $\C$ are epimorphisms, and
    all families $\{A \rto B\}$ consisting of single morphisms with $A\neq
    \initial$ are covering families.
  \item[(D)] If $A,B\neq \initial$ no two morphisms $A \rto C$ and $B \rto C$ are
    disjoint.
  \end{itemize}
\end{definition}
Morally speaking, in this case $\C$ behaves like the assembler $\S_G$, which has
a single noninitial object $S$ with a group of automorphisms acting on it.  If
for a noninitial $A$ we let $\C_A$ be the full subassembler of $\C$ containing
all objects $B$ for which there exists a morphism $B \rto A$, then by
Theorem~\ref{thm:subassemb} the inclusion $\C_A \rto \C$ induces an equivalence
on $K$-theory.  Thus $\C$ is ``homogeneous'' in a certain sense.  By axiom (Ep)
every object in $\W(\C)$ has a morphism to an object each of whose components is
equal to $S$, and by axiom (D) all morphisms preserve the cardinality of the
indexing set.  Thus the higher homotopical structure of $K(\C)$ must come from
``partially defined automorphisms'' of $S$, which are given by zigzags of
morphisms in $\W(\C)$.  In this section we show that this intuition is correct
by showing that there exists a group $G$ and a morphism of assemblers $\C \rto
\S_G$ which induces an equivalence on $K$-theory.

\begin{definition}
  Let $\C$ be an epimorphic assembler with a sink.  The \textsl{group $G$
    associated to $\C$} is defined as follows.  The elements of $G$ are
  equivalence classes of diagrams \smorpair{A}{}{} in $\C$, where $A$ is any
  noninitial object.  We define $\bigg(\smorpair{A}{f_1}{f_2}\bigg) \sim
  \bigg(\smorpair{B}{g_1}{g_2}\bigg)$ if there exists an object $C$ that fits
  into a diagram
  \begin{general-diagram}{1em}{3em}
    { & B \\ S & C & S \\ & A \\};
    \to{1-2}{2-1}_{g_1} \to{1-2}{2-3}^{g_2}
    \to{3-2}{2-1}^{f_1} \to{3-2}{2-3}_{f_2}
    \to{2-2}{1-2} \to{2-2}{3-2}
  \end{general-diagram}
  which commutes.  The multiplication on $G$ of $\bsmorpair{A}{f_1}{f_2}$ and
  $\bsmorpair{B}{g_1}{g_2}$ is represented by the composition down the
  left and the right in the diagram
  \begin{squisheddiagram}
    { && X  \\
      & A & & B \\
      S && S && S \\};
    \to{1-3}{2-2}_{g_1'} \to{1-3}{2-4}^{f_2'}
    \to{2-2}{3-1}_{f_1} \to{2-2}{3-3}^{f_2} \to{2-4}{3-3}_{g_1} \to{2-4}{3-5}^{g_2}
  \end{squisheddiagram}
  where the middle square is any completion of $f_2$ and $g_1$ to a commutative
  square where $X$ is noninitial (which exists by axiom (D)).  The identity in
  $G$ is represented by two copies of the identity morphism $1_S$; the inverse
  of $\bsmorpair{A}{f}{g}$ is $\bsmorpair{A}{g}{f}$.
\end{definition}

To see that the product in $G_\C$ is well-defined, first suppose that there
exists a diagram
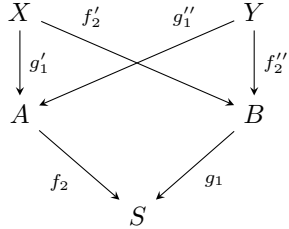
\begin{diagram}[3em]
  { X & & Y \\ A && B \\ & S \\};
  \to{1-1}{2-1}^{g_1'} \to{1-1}{2-3}+{above left=1em}{f_2'}
  \to{1-3}{2-1}+{above right=1em}{g_1''} 
  \to{1-3}{2-3}^{f_2''} 
  \to{2-1}{3-2}_{f_2} \to{2-3}{3-2}^{g_1}
\end{diagram}
which gives us two different completions.  Let $Z$ be noninitial with morphisms
$\alpha:Z \rto X$ and $\beta:Z \rto Y$ which complete $f_2'$ and $f_2''$ to a
commutative square.  Then we have
\begin{align*}
  \makeshort[.6]{(Z \rto^\alpha X \rto^{g_1'} A \rto^{f_2} S)} &= \makeshort[.6]{(Z
    \rto^\alpha X \rto^{f_2'} B \rto^{g_1} S)} \\
  &= \makeshort[.6]{(Z \rto^\beta Y \rto^{f_2''} B \rto^{g_1} S)} \\
  &= \makeshort[.6]{(Z \rto^\beta Y \rto^{g_1''} A \rto^{f_2} S).}
\end{align*}
As all morphisms in $\C$ are monic, this means that $\alpha$ and $\beta$ also
complete $g_1'$ and $g_1''$ to a commutative square, so we have a diagram
\begin{squisheddiagram}[3em]
  { && X & \\ 
    S & A & Z & B & S\\
    && Y  \\};
  \to{1-3}{2-2}_{g_1'} \to{1-3}{2-4}^{f_2'}
  \to{3-3}{2-2}^{g_1''} \to{3-3}{2-4}_{f_2''}
  \to{2-3}{1-3}^{\alpha} 
  \to{2-3}{3-3}^{\beta}
  \to{2-2}{2-1}_{f_1} \to{2-4}{2-5}^{g_2}
\end{squisheddiagram}
and thus the two different representatives represent the same class and the
product in $G$ is well-defined.  Associativity follows analogously.

\begin{definition} \label{def:sinkfam} Let $\C$ be an epimorphic assembler with a
  sink, and let $G$ be the group associated to $\C$.  Let $\F$ be a choice of a
  morphism $f_A:A \rto S$ for each object $A$ in $\C$, with the assumption that
  $f_S = 1_S$.  Then we can define a morphism of assemblers $\pi_{\F}:\C \rto
  \S_G$ by
  \[\pi_\F(A) = * \qquad\hbox{and}\qquad \pi_{\F}(g:A \rto B) =
  \bsmorpair{A}{f_A}{f_Bg}.\]
\end{definition}

\begin{theorem} \label{thm:SEp} Let $\C$ be an epimorphic assembler with a sink,
  and let $G$ be the group associated to $\C$. Then for every choice of family
  $\F = \{f_A:A \rto S\}$ the morphism of assemblers $\pi_{\F}:\C\rto \S_G$
  induces an equivalence on $K$-theory.
\end{theorem}

\begin{proof}
  Let $\pi = \pi_\F$.  The functor $\pi$ is continuous by definition.  Thus to
  see that this is a valid morphism of assemblers we need to check that it
  preserves disjointness, or, equivalently, that no two morphisms between
  noninitial objects in $\C$ are disjoint.  This is true by property (D).

  We now need to check that this induces an equivalence on $K$-theory.  By
  Lemma~\ref{lem:inductive-step} it suffices to show that the map $\W(\C)\rto
  \W(\S_G)$ induces a homotopy equivalence on geometric realization.

  We us Quillen's Theorem A.  We need to show that for all $Y$ in $\W(\S_G)$,
  $\W(\pi)/Y$ is contractible.  But $Y = \{*\}_{i\in I}$ and $\W(\pi)/Y \simeq
  (\W(\pi)/\{*\})^{|I|}$, so it suffices to show that $\D = \W(\pi)/\{*\}$ is
  contractible.  An object of $\D$ is a pair $(A,g)$ with $A$ in $\C$ and $g\in
  G$. We claim that $\D$ is a preorder.  Indeed, suppose that we have two
  objects $(A,g)$ and $(B,h)$.  A morphism $\{A\}_{\{*\}} \rto \{B\}_{\{*\}}$ in
  $\W(\C)$ is a morphism $f:A \rto B$ in $\C$; thus a morphism $ (A,g)\rto
  (B,h)$ is an $f$ in $\C$ such that $\pi(f) = h^{-1}g$.  Therefore $\D$ is a
  preorder exactly when $\pi$ is faithful.  To show that $\pi$ is faithful we
  need to show that if $r,s\in \C(A,B)$ are such that there exists a diagram
  \begin{general-diagram}{2em}{3em}
    { & A \\ S & C & B & S \\ & A\\};
    \to{1-2}{2-1}+{above}{f_A} \to{1-2}{2-3}+{above}{r} \to{3-2}{2-1}+{below}{f_A}
    \to{3-2}{2-3}+{below}{s} 
    \to{2-2}{1-2}+{right=-.2ex}{h_1} \to{2-2}{3-2}+{right=-.2ex}{h_2} \to{2-3}{2-4}^{f_B}
  \end{general-diagram}
  then $r=s$.  As $f_A$ is monic, it follows that $h_1 = h_2$.  As $h_1$ is epic,
  it follows that $r=s$, so we see that $\pi$ is faithful.
  
  As $\D$ is a preorder, to show that it is contractible by
  Lemma~\ref{lem:cofiltpre} it suffices to show that it is cofiltered.  Let
  $(A,g)$ and $(B,h)$ be two objects in $\D$.  An object $(C,k)$ above them is a
  pair of morphisms $f:C\rto A$ and $f':C\rto B$ in $\C$, such that $\pi(f) =
  gk^{-1}$ and $\pi(f') = hk^{-1}$; thus we want to find an object $C$ in $\C$ and
  a pair of morphisms $f:C\rto A$ and $f':C\rto B$ such that $g^{-1}\pi(f) =
  h^{-1}\pi(f')$.

  Pick a representative $\smorpair{X}{g_1}{g_2}$ for $g$ and a representative
  $\smorpair{Y}{h_1}{h_2}$ for $h$.  Pick commutative squares
  \begin{diagram}
    { X' & A && Y' & B \\
      X & S && Y & S \\};
    \arrowsquare{g'_2}{f_A'}{f_A}{g_2}
    \to{1-4}{1-5}^{h_2'}
    \to{1-4}{2-4}_{f_B'}
    \to{1-5}{2-5}^{f_B}
    \to{2-4}{2-5}^{h_2}
  \end{diagram}
  with noninitial $X'$ and $Y'$; these exist because no two morphisms in $\C$
  are disjoint.  Let $C$ be any completion of \[X' \rto^{f_A'} X \rto^{g_1} S
  \lto^{h_1} Y \lto^{f_B'} Y'\] to a square, with morphisms $\alpha: C \rto X'$
  and $\beta: C \rto Y'$.  Define $f$ and $f'$ by 
  \[C \rto^\alpha X' \rto^{g_2'} A \qquad f': C
  \rto^\beta Y' \rto^{h_2'} B.\]  We want to show that $f$ and $f'$ satisfy the
  desired conditions.  We have the following diagram, where the starred square
  may not commute:
  \begin{general-diagram}{2em}{3em}
    {&& X' \\
      & X && A \\
      S && S &\star& C \\
      & Y && B \\
      && Y' \\};
    \to{2-2}{3-1}_{g_1} \to{2-2}{3-3}^{g_2}
    \to{4-2}{3-1}^{h_1} \to{4-2}{3-3}_{h_2}
    \to{1-3}{2-2}_{f_A'} \to{1-3}{2-4}!{g_2'}
    \to{5-3}{4-2}^{f_B'} \to{5-3}{4-4}!{h_2'}
    \to{2-4}{3-3}_{f_A} \to{4-4}{3-3}^{f_B}
    \arrow{out=90,in=0,->}{3-5}{1-3}_{\alpha}
    \arrow{out=-90,in=0,->}{3-5}{5-3}^\beta
    \to{3-5}{2-4}!f \to{3-5}{4-4}!{f'}
    
  \end{general-diagram}
  Then $g = \bsmorpair{X}{g_1}{g_2} = 
  \bsmorpair{C}{g_1f_A'\alpha}{g_2f_A'\alpha} = 
  \bsmorpair{C}{g_1f_A'\alpha}{f_Ag_2'\alpha} =
  \bsmorpair{C}{g_1f_A'\alpha}{f_Af}$, and thus 
  \[g^{-1}\pi(f) = \bsmorpair{C}{f_Af}{g_1f_A'\alpha}
  \bsmorpair{C}{f_C}{f_Af} =
  \bsmorpair{C}{f_C}{g_1f_A'\alpha}.\]
  Analogously,
  \[h^{-1}\pi(f') = \bsmorpair{C}{f_C}{h_1f_B'\beta}.\] But since the outside of
  the diagram commutes, $h_1f_B'\beta = g_1f_A'\alpha$, and
  $g^{-1}\pi(f)=h^{-1}\pi(g')$.  Thus $\D$ is cofiltered and by Quillen's
  Theorem A we have that $|\W(\C)|\rto |\W(\S_G)|$ is a weak equivalence.
\end{proof}

One important thing to observe is that while $K(\pi_\F)$ is an equivalence this
equivalence is not canonical: we may get a different equivalence for each choice
of morphisms to the sink.  We cannot always find a relation between two
different choices of $\F$, but in nice geometric cases we can:

\begin{proposition} \label{prop:differentF}
  Suppose that $\F$ and $\F'$ are two choices of families as in
  Definition~\ref{def:sinkfam}, and let $\psi:\C \rto \C$ be an automorphism
  such that $\psi(S) = S$.  If there exists a family of isomorphisms
  $\{\varphi_A:A \rto \psi(A)\,|\, A\in \C\}$ such that for all $A\neq S$ the
  diagrams
  \begin{diagram}
    { A & \psi(A) & \qquad & A & \psi(A) \\ 
      &   S  &        & S & S \\};
    \to{1-1}{1-2}^{\varphi_A} \to{1-4}{1-5}^{\varphi_A} 
    \to{1-1}{2-2}_{f'_A} \to{1-2}{2-2}^{f_{\psi(A)}} \to{1-4}{2-4}_{f_A}
    \to{1-5}{2-5}^{f_{\psi(A)}} \to{2-4}{2-5}^{\varphi_S}
  \end{diagram}
  commute, then
  \[\pi_{\F'} = \Phi \pi_{\F},\]
  where $\Phi: \S_G \rto \S_G$ is the morphism of assemblers induced by
  conjugation by $\bsmorpair{S}{1}{\varphi_S}$ in $G$.  
\end{proposition}

\begin{proof}
  The desired formula holds tautologically on objects, so we simply need to
  check that it holds on morphisms.  We have
  \begin{eqnarray*}
    \pi_{\F'}(g: A \rto B) &=& \bsmorpair{A}{f'_A}{f'_Bg} =
    \bsmorpair{A}{f_{\psi(A)}\varphi_A}{f_{\psi(B)}\varphi_B g} =
    \bsmorpair{A}{\varphi_S f_A}{\varphi_Sf_Bg} \\
    &=& \bsmorpair{S}{\varphi_S}{1} \bsmorpair{A}{f_A}{f_Bg} \bsmorpair{S}{1}{\varphi_S}.
  \end{eqnarray*}
\end{proof}

\begin{proposition} \label{prop:restricting} Let $\C$ be an epimorphic assembler
  with sinks and let $U$ be a noninitial object of $\C$.  Let $\C_U$ be the full
  subcategory of $\C$ of those objects $A$ such that $\Hom_\C(A,U) \neq \eset$;
  this is again an epimorphic assembler with sink $U$.  Let $\F=\{f_A:A \rto
  S\}$ be a family of morphisms to the sink $S$ of $\C$, and let $\F'=\{f'_A:A
  \rto U\}$ be a family of morphisms to $U$ in $\C_U$ which satisfy $f_A =
  f_Uf'_A$ for all $A$ in $\C_U$.  Let $G$ and $G'$ be the groups associated to
  $\C$ and $\C_U$, respectively, and let $\varphi:G \rto G'$ be defined by
  \[\varphi\Big[\morpair{U}{A}{f}{g}\Big] = \bsmorpair{A}{f_Uf}{f_Ug}.\]
  Then $\varphi$ is an isomorphism of groups and the diagram
  \begin{diagram}
    { \C_U & \C \\ \S_{G'} & \S_G \\};
    \arrowsquare{}{\pi_{\F'}}{\pi_\F}{\S_\varphi}
  \end{diagram}
  commutes.
\end{proposition}

\begin{proof}
  The fact that the diagram commutes follows directly from the definitions.  To
  check that $\varphi$ is an isomorphism, note that it is injective because
  $\C_U$ contains all objects above $U$, and it is surjective because any
  representing pair can be modified to factor through $f_U$.
\end{proof}

\section{Applications of the main theorems} \label{sec:moreex}

In this section we give two applications of the theorems in
Section~\ref{sec:equiv} and of Theorem~\ref{thm:cofibercalc}.  Although
Theorem~\ref{thm:cofibercalc} has not yet been proved, the proof is technical,
long and not terribly illuminating, so we prefer to defer the proof until after
some interesting applications of the theorem are illustrated.  For details of
the proof, see Sections~\ref{sec:cofiber} and \ref{sec:proof2}.

For ease of reading, we reproduce the theorem here.  Recall that for an object
$A$ of $\C$, we say that $\C$ \textsl{has complements for $A$} if any morphism
$A \rto B$ is in a finite disjoint covering family of $B$.
\begin{theorem:cofibercalc}
 Let $\D$ be a subassembler of $\C$
  such that $\D$ is a sieve in $\C$ and such that $\C$ has complements for all
  objects of $\D$.  Then
  \[K(\D) \rto K(\C) \rto K(\C\bs\D)\]
  is a cofiber sequence.
\end{theorem:cofibercalc}

\subsection{The Grothendieck ring of varieties} \label{sec:gr} 

Let $k$ be a field.  We define $\V_k$ to be the category whose objects are
$k$-varieties (by which we mean reduced separated schemes of finite type over
$k$), and whose morphisms are finite composites of open embeddings and closed
embeddings.  The topology on $\V_k$ is generated by the coverage $\{Y \rto X,
X\bs Y \rto X\}$ for closed embeddings $Y \rto X$.  (For background on
coverages, see for example \cite{johnstone02v2}.)  Then $\V_k$ is an assembler.
$K_0(\V_k)$ is the abelian group underlying the Grothendieck ring of varieties.

\begin{example} \label{ex:pointcount} Let $k$ be a finite field and let $\F$ be
  the assembler of finite sets and injections.  Let $L$ be a finite algebraic
  extension of $k$.  Then we have a morphism of assemblers $c_L: \V_k \rto \F$
  given by $c_L(X) = X(L)$.  Thus point counting is an example of a morphism of
  assemblers, and taking $K$-theory gives us a ``derived'' notion of point
  counting: a map $K(\V_k) \rto \S$.  (The $K$-theory of $\F$ is equivalent to
  the sphere spectrum by Theorem~\ref{thm:subassemb} for the inclusion $\S \rto
  \F$ taking $*$ to a singleton set.)
\end{example}

Let $\V_k^{(n)}$ be the full subassembler of $\V_k$ consisting of all varieties
of dimension at most $n$.  The goal of this section is to prove the following
theorem: 
\begin{theorem}\label{eq:SS}
  Let $B_n$ be the set of birational isomorphism classes of varieties of
  dimension $n$.  For any variety $X$ over $k$, let $k(X)$ be its function
  field. Then
  \[\hocofib \big(K(\V_k^{(n-1)}) \rto K(\V_k^{(n)})\big) \simeq \bigvee_{[X]\in
    B_n} \Sigma_+^\infty B\Aut\, k(X).\] 
\end{theorem}
This theorem gives us a spectral sequence
\[E^1_{p,q} = \bigoplus_{[X]\in B_n} \pi_p^s B\Aut\, k(X) \Rto K_p K(\V_k)\]
that is a key tool in the proof of Theorem~\ref{thm:kernelL}.

Note that as a subassembler of $\V_k^{(n)}$, the assembler $\V_k^{(n-1)}$
satisfies the conditions of Theorem~\ref{thm:cofibercalc}.  Thus by
Theorem~\ref{thm:cofibercalc},
\[\hocofib \big(K(\V_k^{(n-1)}) \rto K(\V_k^{(n)})\big) \simeq K(\V^{(n)}_k \bs
\V_k^{(n-1)}).\] The rest of this section is dedicated to analyzing the homotopy
type of $K(\V^{(n)}_k \bs \V_k^{(n-1)})$.

For a fixed irreducible variety $X$ of dimension $n$, let $\C_X$ be the full
assembler of $\V_k^{(n)}\bs\V_k^{(n-1)}$ with objects varieties $Y$ of dimension
$n$ for which there exists a morphism $Y \rto X$ (although we do not include
this morphism in the data).  For a nonempty variety $Y$ the only finite disjoint
covering families are of the form $\{Y' \rto Y\}$ for a noninitial $Y'$; any
morphism between noninitial objects gives a finite disjoint covering family.

Let $B_n$ be as above, and pick a representative $X_\alpha$ for each $\alpha\in
B_n$.  Let $\C = \bigvee_{\alpha\in B_n} \C_{X_\alpha}$.  We claim that $\C$ is
a full subassembler of $\V_k^{(n)} \bs \V_k^{(n-1)}$ which satisfies the
conditions of Theorem~\ref{thm:subassemb}, and thus $K(\V^{(n)}\bs
\V^{(n-1)})\simeq K(\C)$.  The subassemblers $\C_{X_\alpha}$ do not intersect
inside $\V_k^{(n)} \bs \V_k^{(n-1)}$, since if some noninitial $U$ were inside
two of them then there would be morphisms $U \rto X_\alpha$ and $U \rto
X_\beta$, which gives a birational isomorphism between $X_\alpha$ and $X_\beta$,
which cannot happen when $\alpha\neq \beta$.  Therefore $\C$ is a subassembler
of $\V_k^{(n)} \bs \V_k^{(n-1)}$, and it remains to check that it is full.
Suppose it were not full, so that there existed some morphism $U_\alpha \rto
V_\beta$ with $U_\alpha\in \C_{X_\alpha}$ and $V_\beta\in \C_{X_\beta}$.  Then
again there would exist a morphism $U_\alpha \rto V_\beta \rto X_\beta$, and
$X_\alpha$ would be birational to $X_\beta$, a contradiction.

Now we claim that $\C$ satisfies the conditions of Theorem~\ref{thm:subassemb}
inside $\V_k^{(n)} \bs \V_k^{(n-1)}$. Suppose that $Y$ is any variety of
dimension $n$.  It can be written as $Y = \bigcup_{i=1}^\ell Y_i$, where the
$Y_i$ are the irreducible components of $Y$; we also assume without loss of
generality that there exists $1\leq m \leq \ell$ such that $\dim Y_i = n$ if $i
\leq m$ and $\dim Y_i < n$ otherwise.  Since $B_n$ is all birational isomorphism
classes of irreducible varieties, $[Y_i] = [X_{\alpha_i}]$ for some $\alpha_i\in
B_n$.  Thus there is some open subset $U_i$ of $Y_i$ which is disjoint from all
other $Y_j$ and such that there exists a morphism $U_i \rto X_{\alpha_i}$.
Therefore $U_i\in \C$.  The family $\{U_i \rto Y\}$ is a covering family in
$\V_k^{(n)} \bs \V_k^{(n-1)}$, since it can be completed to the finite disjoint
covering family 
\[\makeshort{\left\{U_i \rto Y, \bigcup_{i=1}^m (Y_i \bs U_i) \cup \bigcup_{i=m+1}^\ell Y_i \rto
Y\right\}}\] in $V_k^{(n)}$.

From this we can conclude that 
\[K(\V_k^{(n)}\bs \V_k^{(n-1)}) \simeq K(\C)
\simeq \bigvee_{[X_\alpha]\in B_n} K(\C_{X_\alpha}).\]  Thus we can now restrict
our attention to determining the homotopy type of $K(\C_X)$ for any variety $X$.

In the following analysis, we will fix the model of $\C_X$: we think of the
objects of $\C_X$ as subsets of $X(\bar k)$; then for every $Y$ in $\C_X$ there
is a preferred morphism $\iota_Y: Y \rto X$ which is an inclusion of points.
With the choice of this model we have chosen ``coordinates'' for all objects of
$\C_X$ simultaneously; by Corollary~\ref{cor:equiv-ass} this choice does not
affect the homotopy type of the $K$-theory of the assembler.  

We prove a slightly more general theorem than is needed to complete the proof of
Theorem~\ref{eq:SS}, as it will be necessary for the proof of
Theorem~\ref{thm:kernelL}.

\begin{theorem}
  Let $X$ be an irreducible variety of dimension $n$, and let $\C_X$ be the
  assembler whose objects are subvarieties of $X$ of dimension $n$, modeled as
  algebraic subsets of $X(\bar k)$ defined over $k$.
  \begin{itemize}
  \item[(1)] $K(\C_X) \simeq \Sigma^\infty_+ B\Aut\, k(X)$.  
  \item[(2)] For any variety $X$ there is a morphism of assemblers $\C_X \rto
    \C_{X\times \A^1}$ which takes a variety $Y$ to $Y\times \A^1$ and a
    morphism $f:Y \rto Y'$ to $f\times 1_{\A^1}$.  Write $\Aut k(X\times \A^1)$
    as $\Aut k(X)(t)$ for a transcendental $t$; let $\varphi: \Aut k(X) \rto
    \Aut k(X)(t)$ be the homomorphism that includes $\Aut k(X)$ as those
    automorphisms which fix $t$.  Then the diagram
    \begin{diagram}
      { \C_X & \C_{X\times \A^1} \\ \S_{\Aut\,k(X)} & \S_{\Aut\,k(X\times\A^1)}
        \\};
      \arrowsquare{}{}{}{\varphi}
    \end{diagram}
    commutes.
  \end{itemize}
\end{theorem}

\begin{proof}
  This proof relies on Theorem~\ref{thm:SEp}.  For any variety $Z$ we define the
  family
  \[\F_Z = \{\iota_A: A \rto Z\,|\, A\in \C_Z\}.\]
 
  \noindent
  Proof of (1):  The assembler $\C_X$ satisfies the conditions of
    Theorem~\ref{thm:SEp}, so we get a morphism of assemblers $\pi_{\F_X}: \C_X \rto
    \S_{G}$ which induces an equivalence after applying $K$-theory. In this case
    $G$ is the group $\Aut\, k(X)$, the birational automorphisms of $X$.
  
  \noindent
  Proof of (2): With our definition of $\F_{X\times \A^1}$, for any subvariety $Y$
  of $X$ we have $\iota_{Y\times \A^1} = \iota_Y \times 1_{\A^1}$.  Thus if we define
  \[\varphi\Big[\morpair{X}{A}{f}{g}\Big] \defeq \Big[
  \begin{tikzpicture}[baseline=(A.base)] \node[anchor=east]%
    (A) at (0,0) {$X\times \A^1$}; \node[anchor=west] (B) at (3em,0) {$A\times \A^1$};%
    \diagArrow{<-, bend left}{A}{B}!{f\times 1_{\A^1}} \diagArrow{<-, bend %
      right}{A}{B}!{g\times 1_{\A^1}}%
  \end{tikzpicture}\Big]\]
  the given diagram commutes.  This $\varphi$ is injective and hits all partial
  automorphisms of $X\times \A^1$ which are defined ``fiberwise'' on some
  $U\times \A^1$, which gives the desired algebraic description.
\end{proof}

\subsection{Classical scissors congruence} \label{sec:polytope}

There are two standard approaches to scissors congruence for polytopes.  The
first focuses on one dimension at a time by saying that two $n$-polytopes are
``interior disjoint'' if their intersection contains no $n$-polytopes.  We can
then define scissors congruence of $n$-polytopes in $\R^\infty$ in the following
manner.  An $n$-simplex is defined to be the convex hull of $n+1$ points in
general position; an $n$-polytope is a finite union of $n$-simplices.  We then
say that two $n$-polytopes $P$ and $Q$ are scissors congruent if
\begin{itemize}
\item we can write $P = \bigcup_{i=1}^m P_i$ and $Q = \bigcup_{i=1}^m Q_i$ such
  that $P_i\cong Q_i$, and
\item $P_i \cap P_j$ and $Q_i \cap Q_j$ contain no $n$-simplices for $i\neq j$.
\end{itemize}
This is the approach considered in the work of Dupont and Sah (see \cite{sah79}, 
\cite{dupontsah82} and \cite{dupont01}).  We write the scissors congruence
group of $n$-polytopes defined through this approach as $\mathcal{P}_n$.

An alternate approach considers all intersections of all dimensions.  Here we
want to be able to decompose polytopes into completely disjoint sets.  The
simplest way to write this down is to say that an $n$-simplex is the interior
(in the $n$-space spanned by the points) of the convex hull of $n+1$ points, and
that an $n$-polytope is a finite union of simplices $X_i,\ldots,X_k$, where
$\dim X_i \leq n$ and for some $i_0$, $\dim X_{i_0} = n$.  Thus for example, a
closed interval is a $1$-polytope because a point is a $0$-simplex, an open
interval is a $1$-simplex, and a closed interval is the union of an open
interval and two points.  We then say that two polytopes $P$ and $Q$ are
scissors congruent if
\begin{itemize}
\item we can write $P = \bigcup_{i=1}^m P_i$ and $Q = \bigcup_{i=1}^m Q_i$ such
  that $P_i \cong Q_i$, and
\item $P_i \cap P_j = Q_i \cap Q_j = \eset$.
\end{itemize}
This is an approach which is analogous to the scissors congruence of definable
sets (considered in, for example, \cite{dries98}).  When the only isometries
allowed are translations these groups were analyzed by McMullen in
\cite{mcmullen89}; the more general case was studied by Goodwillie in
\cite{goodwillie14}.  We denote the scissors congruence group of such polytopes
by $\sP$.

As dimension is a scissors-congruence invariant, we have the following
observation: 
\begin{observation} \label{obs:polytope}
  The group $\sP$ is filtered by dimension.  More formally, if we define
  $\sP_n$ to be the scissors congruence group of $m$-polytopes with
  $m\leq n$ then we have a sequence of homomorphisms
  \[\makeshort{0 = \sP_{-1} \rto \sP_0 \rcofib
    \sP_1 \rto\cdots \rto \sP}.\]
  We have a filtration on $\sP$ where the $n$-th filtered piece is the image of
  $\sP_n$ inside $\sP$.
\end{observation}
However, it is unclear whether or not higher dimensional polytopes can induce
relations between lower-dimensional ones, so we do not know whether $\sP_n$ is
the $n$-th graded piece of $\sP$.  We want to compute the associated graded
spectrum of this filtration, and use it to learn about the structure of $\sP$.

First, we consider the analog of Dupont and Sah's approach.  We define an
assembler $\G_{n}$ whose objects are pairs $(U,P)$ where $P$ is an $n$-polytope
in $\R^\infty$ and $U$ is the smallest affine subspace containing $P$.  A
morphism $(U,P) \rto (V,Q)$ is an isometry $\varphi:U \rto V$ such that
$\varphi(P) \subseteq Q$.  The Grothendieck topology is defined by defining a
family 
\[\big\{\varphi_\alpha:(U_\alpha,P_\alpha) \rto (U,P)\big\}_{\alpha\in A}\] to be a
covering family if $P = \bigcup_{\alpha\in A} \varphi_\alpha(P_\alpha)$.

Now we construct the analog of McMullen and Goodwillie's approach.  We define an
assembler $\gG$ to have as objects pairs $(U,P)$, where $P$ is an $n$-polytope
in $\R^\infty$ and $U = \mathrm{span}\, P$.  Morphisms $(U,P) \rto (V,Q)$ of
pairs consist of an isometry $\varphi:U\rto V$ such that $\varphi(P) \subseteq
Q$.  Note that these are almost the same definitions as in the Dupont and Sah
approach; the difference is that here we take all $n$, and that our polytopes
have a slightly different definition, in that they do not necessarily contain
their boundaries.  The Grothendieck topology is defined analogously to the
Grothendieck topology on $\G_{n}$.

\begin{remark}
  Both of these constructions can be done using only simplices, without
  considering general polytopes.  By Theorem~\ref{thm:subassemb} this approach
  gives an equivalent $K$-theory.  However, as the traditional approach uses
  polytopes, we do as well.
\end{remark}

Morphisms in $\gG$ are defined in exactly the same way as the morphisms in
$\G_{n}$, so that we have an inclusion functor $\G_{n} \rto \gG$.  From the
definition of the topologies this is a continuous functor.  However, this is not
a morphism of assemblers as it does not preserve disjointness: given any two
$n$-polytopes with a nonempty measure-$0$ intersection, their pullback (over
their union) is $\initial$ in $\G_{n}$ but not in $\gG$.

Na\"ively we might think that $\gG$ is a product of $\G_{n}$'s for different
$n$, but the above observation about the functor $\G_{n} \rto \gG$ means that
this is not the case.  Even playing around with small values of $n$ shows that
something is different, as in $\gG$ we have an extra invariant on polytopes: the
Euler characteristic.  (For more on this, see \cite{dries98}.)

\begin{proposition} \label{prop:grpoly} Let $\gG_m$ be the full subcategory of
  $\gG$ containing all pairs $(U,P)$ where $P$ is an $i$-polytope with $i\leq
  m$.  This is a subassembler, and we get a sequence
  \[\makeshort[.6]{* = \gG_{-1} \rcofib^{\iota_0} \gG_{0} \rcofib^{\iota_1} \gG_{1} \rcofib^{\iota_2}
    \cdots \rcofib \gG, }\]
  where $*$ is the assembler containing only the initial object.  We have
  \[\hocofib K(\iota_n) \simeq K(\G_{n}).\]
  Thus $\bigvee_{n=0}^\infty K(\G_{n})$ is the associated graded spectrum of
  $K(\gG)$.
\end{proposition}

\begin{proof}
  Since $\gG_{n-1}$ satisfies the conditions of Theorem \ref{thm:cofibercalc}
  it follows that \[\hocofib K(\iota_n) \simeq K(\gG_{n} \bs
  \gG_{n-1}).\]  Therefore it suffices to construct a morphism of assemblers
  $F:\G_n \rto \gG_n \bs \gG_{n-1}$ which induces an equivalence on $K$-theory.
  Let $F':\G_{n} \rto \gG_{n}$ take each pair $(U,P)$ to the pair $(U, \mathring
  P)$, where $\mathring P$ is the open interior of $P$.  This functor preserves
  pullbacks but is not continuous.  We claim that after mapping to $\gG_{n} \bs
  \gG_{n-1}$ this functor induces an equivalence of $K$-theories.
  
  Let $F$ be the composite
  \[\G_n \rto^{F'} \gG_{n} \rto^\pi  \gG_{n}\bs\gG_{n-1},\]
  where $\pi$ is the canonical morphism $\gG_{n} \rto^\pi \gG_{n}\bs\gG_{n-1}$.
  As both $\pi$ and $F'$ preserve disjointness and the initial object, so does
  $F$.  Thus to show that $F$ is a valid morphism of assemblers it suffices to
  show that it preserves covering families.  A covering family in $ \G_n$ is a
  family $\{f_\alpha: P_\alpha \rto P\}_{\alpha\in A}$ such that
  $\bigcup_{\alpha\in A} f_\alpha(P_\alpha) = P$.  For each $\alpha$ we can
  write $P_\alpha = \mathring P_\alpha \cup \partial P_\alpha$, where $\partial
  P_\alpha$ is the boundary of $P_\alpha$.  Note that $\partial P_\alpha
  \subseteq \gG_{n-1}$.  We then have a covering family
  \[\{f_\alpha:\mathring P_\alpha \rto P\}_{\alpha \in A} \cup
  \{f_\alpha: \partial P_\alpha \rto P\}_{\alpha\in A}\] in $\gG_{n}$.  Each
  source in the second half is in $\gG_{n-1}$, so it is killed by $\pi$; thus by
  definition of the topology on $\gG_n \bs \gG_{n-1}$ $\{f_\alpha: \circ
  P_\alpha \rto P\}_{\alpha\in A}$ is a covering family, and we see that $F$
  preserves covering families.
  
  Note that $F$ includes $\G_n$ as a subassembler of $\gG_{n}\bs \gG_{n-1}$.
  The objects of $\gG_{n}\bs \gG_{n-1}$ which are hit by $F$ are exactly those
  objects $(U,P)$ such that $\mathring P = P$.  Any object
  $(U,P)$ has a disjoint covering family
  \[\{(\mathrm{span}(\partial P), \partial P) \rto
  (U,P),(\mathrm{span}\,\mathring P, \mathring P) \rto (U,P)\},\] so the family
  $\{(\mathrm{span}\,\mathring P, \mathring P) \rto (U,P)\}$ is a covering
  family of $(U,P)$ in $\gG_{n}\bs \gG_{n-1}$.  The second one of these is in
  the image of $F$ and we see that $F(\G_n)$ satisfies the condition of
  Theorem~\ref{thm:subassemb}.  Thus $K(\G_n) \simeq K(\gG_{n}\bs \gG_{n-1})$.
\end{proof}

As a corollary we have:

\begin{theorem}
  There is a spectral sequence 
  \[E^1_{p,q} = K_p(\G_q) \Rto K_p(\gG).\]
  If we filter $\gG_{n}$ instead of $\gG$ we get the spectral sequence in
  Theorem~\ref{thm:scissors}. 
\end{theorem}

\begin{proof}
  These are the spectral sequences associated to the filtered spectra
  \[K(\gG_{0}) \rto K(\gG_{1}) \rto \cdots \rto K(\gG)\]
  and
  \[K(\gG_{0}) \rto K(\gG_{1}) \rto \cdots \rto K(\gG_{n}).\]
\end{proof}

\section{The cofiber theorem} \label{sec:cofiber}

Our goal in this section is to prove Theorem~\ref{thm:cofiber}.

\begin{theorem:cofiber}
  Let $g:\D_\dot\rto \C_\dot$ be a morphism of
  simplicial assemblers.  There exists a simplicial assembler $(\C/g)_\dot$ with
  a morphism of simplicial assemblers $\iota:\C_\dot \rto (\C/g)_\dot$ such
  that the sequence
  \[K(\D_\dot) \rto^{K(g)} K(\C_\dot) \rto^{K(\iota)} K((\C/g)_\dot)\] is a
  cofiber sequence.
\end{theorem:cofiber}

\begin{definition}
  Suppose that we are given a morphism of simplicial assemblers $g:\D_\dot \rto
  \C_\dot$.  We define the simplicial assembler $(\C/g)_\dot$ by
  \[(\C/g)_n = \C_n \vee ((S^1)_n \wedge \D_n).\] 
  As $\vee$ is the coproduct in $\ASB$, in order to define the face and
  degeneracy maps it suffices to construct them on each component separately.
  On $\C_n$, the face and degeneracy maps are induced by the simplicial
  structure on $\C_\dot$; on $(S^1)_n \wedge \D_n$ all structure maps other than
  $d_0$ are induced by the simplicial structure maps of $S^1$ and $\D_\dot$.  
  Recall that $\ps{n} = \{*,1,\ldots,n\}$.  Note that $(S^1)_n \cong \ps{n} \cong
\ps{1} \vee (\ps{n-1})$.  We define $d_0$ to be the composite
  \[\begin{array}{l}
    C_n \vee (S^1)_n\wedge \D_n \rto^{\cong} C_n \vee (\ps{1}\vee (S^1)_{n-1})
    \wedge \C_n \vee \D_n \\
    \qquad \rto^{\cong} C_n \vee \D_n \vee
    ((S^1)_{n-1} \wedge \D_n) \ms[1.9]{\rto^{gd_0 \vee (1\wedge d_0)} \C_{n-1} \vee ((S^1)_{n-1}
      \wedge \D_{n-1})}.
  \end{array}\]
\end{definition}

The simplicial assembler $(\C/g)_\dot$ comes with a natural inclusion
$\iota:\C_\dot \rto (\C/g)_\dot$ and a natural projection $\pi_\D:(\C/g)_\dot
\rto S^1\wedge \D_\dot$.  Theorem~\ref{thm:cofiber} states that 
\[K(\D_\dot) \rto^{K(g)} K(\C_\dot) \rto^{K(\iota)} K((\C/g)_\dot)\]
is a cofiber sequence.

We prove this using an approach analogous to the one used by Waldhausen for
\cite[Propositions 1.5.5, 1.5.6]{waldhausen83}.  However, as the structure of
our spectra is much simpler than his, the proof turns out to be much more basic.
The key point is the additivity theorem: all we need is the two-term case of
Proposition~\ref{prop:Wprop}(3) to show that $\W(\C\vee\C)$ is equivalent to
$\W(\C)^2$.

Most of the work of the proof is contained in the following lemma:
\begin{lemma} \label{lem:cofiber-fibseq}
  For all $k\geq 1$, the sequence
  \[
    K(\C_\dot)_k \rto K((\C/g)_\dot)_k \rto K(S^1\wedge \D_\dot)_k
    \]
  is a homotopy fiber sequence.
\end{lemma}

\begin{proof}
  Each term in this sequence is the geometric realization of a $k$-simplicial
  category.  As the geometric realization is a diagonal, we can turn this into a
  sequence of bisimplicial sets by taking a partial diagonal.  By choosing this
  partial diagonal appropriately, we get a sequence of bisimplicial sets which
  at level $(m,n)$ the sequence is equal to
  \begin{align*}
  \W((S^{k-1})_n \wedge (S^1)_m \wedge \C_n) &\rto \W((S^{k-1}_n) \wedge ((S^1)_m) \wedge
  (\C_n \vee ((S^1)_n\wedge \D_n))) \\ & \rto \W((S^{k-1})_n \wedge (S^1)_m \wedge
  ((S^1)_n \wedge \D_n)).
\end{align*}
By generalizing \cite[Lemma 5.2]{waldhausen78} to trisimplicial sets, to prove
the lemma it suffices to show that:
\begin{itemize}
\item[(1)] the composition is constant,
\item[(2)] for all $n$ and all pointed sets $X$, $\big|[m] \rgoesto \W(X \wedge
  (S^1)_m \wedge (S^1)_n \wedge \D_n)\big|$ is connected, and
\item[(3)] for all $n$ and all pointed sets $X$,
  \begin{align*}
    &\makeshort[1.2em]{\Big|[m]\rgoesto \W(X \wedge (S^1)_m \wedge \C_n)\Big|} \\
    &\qquad\rto \makeshort[1.2em]{\Big|[m] \rgoesto \W(X \wedge (S^1)_m \wedge (\C_n
      \vee ((S^1)_n\wedge \D_n)))\Big|} \\ &\qquad\rto \makeshort[1.2em]{\Big|[m]\rgoesto
      \W(X \wedge (S^1)_m \wedge (S^1)_n \wedge \D_n})\Big|
  \end{align*}
  is a homotopy fiber sequence.
\end{itemize}
In (2) and (3), we really only care about $X = (S^{k-1})_n$.  (1) is clear
because the map of simplicial assemblers composes to a constant map.  (2)
follows because the $0$-simplices arise from the trivial category.  Write $\C' =
X\wedge \C_n$ and $\D' = X\wedge (S^1)_n\wedge \D_n$.  There exists a functor
\[\W((S^1)_m\wedge (\C'\vee \D')) \rto \W((S^1)_m\wedge \C')\times \W((S^1)_m
\wedge \D')\]
which takes an object $\SCob{a}{i}$ to the pair
\[\Big( \{A_i\}_{i\in \{i\in I\,|\, A_i\in (S^1)_m \wedge \C'\}},
\{A_i\}_{i\in \{i\in I\,|\, A_i\in (S^1)_m \wedge \D'\}}\Big);\] this is an
equivalence of categories by Proposition~\ref{prop:Wprop}.
This functor fits into a commutative diagram
\begin{diagram}[2em]
  {\W((S^1)_m \wedge \C') & \W((S^1)_m \wedge (\C' \vee
    \D')) & \W((S^1)_m \wedge \D') \\
    \W((S^1)_m \wedge \C') & \W((S^1)_m \wedge \C')\times \W(
    (S^1)_m \wedge \D') & \W((S^1)_m \wedge \D') \\};
  \eq{1-1}{2-1} \eq{1-3}{2-3} 
  \to{1-1.mid east}{1-2.mid west}^{\W(\iota)}
  \to{2-1.mid east}{2-2.mid west}
  \to{2-2.mid east}{2-3.mid west}^{\pi_2}
  \to{1-2}{2-2}^\simeq 
  \to{1-2.mid east}{1-3.mid west}^{\W(\pi_\D)}
\end{diagram}
where the bottom row is a fiber sequence.  The morphisms in the diagram also
commute with the simplicial structure maps, and thus assemble into a diagram of
simplicial categories.  Note that it is very important in this context that we
are fixing $n$, as if $n$ varies then the middle morphism does not commute with
$d_0$.  Upon taking geometric realization, the middle map turns into a weak
equivalence of simplicial spaces.  As the bottom row stays a fiber sequence even
after geometric realization, the top row is a homotopy fiber sequence.
\end{proof}

We are now ready to prove Theorem~\ref{thm:cofiber}:
\begin{proof}[Proof of Theorem~\ref{thm:cofiber}]
  Consider the commutative diagram
  \begin{diagram}
    { \D_\dot & (\D/1)_\dot & S^1\wedge \D_\dot \\
      \C_\dot & (\C/g)_\dot & S^1\wedge \D_\dot \\};
    \arrowsquare{\iota}{g}{g}{\iota}
    \eq{1-3}{2-3} \to{1-2}{1-3}^{\pi_\D} \to{2-2}{2-3}^{\pi_\D}
  \end{diagram}
  After applying $K$, the top and bottom are levelwise homotopy fiber sequences
  by Lemma~\ref{lem:cofiber-fibseq}, and thus homotopy fiber sequences of
  spectra.  The right-hand map is the identity, so the left-hand square is a
  homotopy pullback square.  $K((\D/1)_\dot)$ is obtained from a simplicial
  shift of $K(S^1\wedge \D)$, so it is contractible.  Thus the composition
  around the bottom is a homotopy fiber sequence.
\end{proof}

This has the following corollary:

\begin{corollary} \label{cor:omega-spectrum}
  $K(\C_\dot)$ is an $\Omega$-spectrum above level $0$.
\end{corollary}

\begin{proof}
  Let $*$ be the assembler with no non-initial objects.  Consider the functor
  $p:\C_\dot \rto *$ which sends all objects of every level of $\C_\dot$ to the
  initial object.  Then $(*/p)_\dot$ has as its $n$-th level $(S^1)_n\wedge \C_n$
  and so $K((*/p)_\dot)_k \cong K(\C_\dot)_{k+1}$.  Therefore by
  Lemma~\ref{lem:cofiber-fibseq} we get a homotopy fiber sequence $K(\C_\dot)_k
  \rto * \rto K(\C_\dot)_{k+1}$.  Thus we get an equivalence $K(\C_\dot)_k
  \simeq \Omega K(\C_\dot)_{k+1}$.
\end{proof}

\section{Proof of Theorem~\ref{thm:cofibercalc}} \label{sec:proof2}


Recall Theorem~\ref{thm:cofibercalc}:
\begin{theorem:cofibercalc}
  For an object $A$ of $\C$, we say that
  $\C$ \textsl{has complements for $A$} if any morphism $A \rto B$ is in a
  finite disjoint covering family of $B$.  Let $\D$ be a subassembler of $\C$
  such that $\D$ is a sieve in $\C$ and such that $\C$ has complements for all
  objects of $\D$.  Then
  \[K(\D) \rto K(\C) \rto K(\C\bs\D)\]
  is a cofiber sequence.
\end{theorem:cofibercalc}

We begin by fixing some notation.  Let $i:\D\rcofib \C$ be the inclusion of
assemblers, and $c: \C \rto \C\bs\D$ be the canonical projection.  
There is a morphism of simplicial assemblers $p:(\C/i)_\dot \rto \C\bs
\D$ given by projecting all copies of $\D$ in $(\C/i)_n$ to the initial object,
and using the canonical morphism $c:\C\rto \C\bs\D$.  

Consider the following commutative diagram of simplicial assemblers:
\begin{squisheddiagram}
  {&& (\C/i)_\dot \\  
    \D & \C \\ 
    && \C\bs\D \\};
  \to{2-1}{2-2}^i \to{2-2}{1-3} \to{2-2}{3-3}^c \to{1-3}{3-3}^p
\end{squisheddiagram}
After applying $K$-theory, by Theorem~\ref{thm:cofiber} the composition along
the top becomes a cofiber sequence.  Thus it suffices to show that $p$ induces
an equivalence of spectra to ensure that the composition along the bottom is
also a cofiber sequence.  By Lemma \ref{lem:inductive-step} it suffices to show
that $\W(p):\W((\C/i)_\dot) \rto \W(\C\bs\D)$ is an equivalence after geometric
realization, which is the statement of the following proposition.

\begin{proposition} \label{prop:cofiblevel1} Suppose that the conditions of
  Theorem~\ref{thm:cofibercalc} hold, and let $i:\D\rto \C$ be the inclusion of
  $\D$ into $\C$.  Then
  \[\big|\W(p)\big|:\big|\W((\C/i)_\dot)\big| \rto \big|\W(\C\bs\D)\big|\]
  is a weak equivalence.
\end{proposition}

The rest of this section is dedicated to the proof of this proposition.  As the
proof is quite long, we start with the setup and an outline of the proof, and
finish the section with a series of lemmas that fill in the details.

Let $P$ be the map $N \W((\C/i)_\dot) \rto N \W(\C\bs\D)$ of bisimplicial sets;
we want to show that it is a weak equivalence.  We assume that the
nerve-direction is horizontal, and the internal $(\C/i)_\dot$-direction is
vertical.  Thus $N \W(\C\bs\D)$ is constant in the vertical direction.  For
conciseness we define $\kappa = \W(c):\W(\C) \rto \W(\C\bs\D)$.

For any subcategory $\C'\subseteq \C$ and any object
$A = \SCob{A}{i}$ in $\W(\C)$ we define
\[A_{\C'} = \{A_i\}_{\{i\in I \,|\, A_i\in \C'\}}.\] 
Observe that while $\C\bs\D$ is not a \textsl{subassembler} of $\C$, it is a
\textsl{subcategory} of $\C$, and we will often be treating it as such.
More generally, for any diagram
\[D = (A_0 \rto^{f_1} A_1 \rto^{f_2} \cdots \rto^{f_n} A_n)\] in $\W(\C)$ we write
\[D_{\C'} = (A_0' \rto A_1' \rto \cdots \rto A_n'),\]
where $A_n' = (A_n)_{\C'}$ and $A_i'$ is the subobject of $A_i$ whose image
is in $A_n'$.  More formally, if we write $A_i = \{A_{ij}\}_{j\in J_i}$ then we
define $A_i' = \{A_{ij}\}_{j\in J_i'}$, with $J_n' = \{j\in J_n \,|, A_{nj}\in
\C'\}$ and $J_i' = f_i^{-1}(J_{i+1}')$.  Note that $D_{\C'}$ is still a diagram
in $\W(\C)$, not $\W(\C')$, since the $A_i$ for $i<n$ may have components which
are not in $\C'$.  Thus for objects $A$, $\kappa(A) = A_{\C\bs\D}$, but this is
not true for morphisms. 

We define
\[Y_n = \big\{(A_0 \rto \cdots \rto A_n)\in N_n\W(\C) \,\big|\, A_n \in
\W(\C\bs\D)\big\}.\] We want to make $Y_\dot$ into a simplicial set. We define
$d_n$ on an $n$-simplex
\[D = (A_0 \rto A_1 \rto \cdots \rto A_n)\] to be
$(A_0\rto \cdots \rto A_{n-1})_{\C\bs \D}$; we let the other face maps and all
degeneracies be inherited from $N \W(\C)$.  This is a well-defined
simplicial set because $\D$ is a sieve in $\C$.  

Define the category $\W(\C,\D)$ to have $\ob \W(\C,\D) = \ob \W(\C\bs\D)$.
Define a morphism
\[f:\SCob{A}{i} \rto \SCob{B}{j} \in \W(\C,\D)\] to be a map of sets $f:I \rto
J$ together with morphisms $f_i:A_i \rto B_{f(i)}$ for $i\in I$ such that
for all $j\in J$ there exists a family $\{D_k \rto B_j\}_{k\in K}$ of morphisms
in $\C$ with $D_k\in \D$ such that 
\[\{f_i:A_i \rto B_j\}_{i\in f^{-1}(j)} \cup \{D_k \rto B_k\}_{k\in K}\] is a
finite disjoint covering family in $\C$.  Then $\W(\C,\D)$ is a subcategory of
$\W(\C\bs \D)$ with inclusion functor
\[I: \W(\C,\D) \rto \W(\C\bs\D).\]

\begin{example}
  Let $\C$ be the assembler whose objects are open, half-open and closed
  segments in $\R$, together with the points of $\R$, and whose morphisms are
  compositions of translations and inclusions.  Let $\D$ be the subassembler
  containing only points.  Then $\C\bs\D$ has as its objects all segments of
  nonzero length, and the covering families in $\W(\C\bs\D)$ are families
  $\{f_i:A_i \rto B\}$ such that $f_i(A_i) \cap f_j(A_j)$ is a finite set.  On
  the other hand, morphisms of $\W(\C,\D)$ consist of those covering families
  whose endpoints do not intersect at all.  Even though these categories are
  different, we will prove in Section~\ref{sec:iota} that they are homotopy
  equivalent.
\end{example}

\begin{example}
  Suppose that $\C$ is the following preorder
  \begin{squisheddiagram}
    { & & B\\
      \initial & A & & D \\
      & & C\\};
    \to{2-1}{2-2} \to{2-2}{1-3} \to{2-2}{3-3} \to{1-3}{2-4} \to{3-3}{2-4}
  \end{squisheddiagram}
  with the only nontrivial covering family being $\{B \rto D, C\rto D\}$.
  Suppose that $\D$ is the full subcategory with $\ob \D = \{\initial, A\}$.
  Then $\W(\C\bs \D)$ has objects triples of integers $(b,c,d)$ which count how
  many copies of $B$, $C$ and $D$ (respectively) the object contains.  There
  exist morphisms $(b,c,d) \rto (b-k,c-k,d+k)$ for all $k \leq \min(b,c)$.
  However, no covering family with objects in $\C\bs\D$ has a completion to a
  finite disjoint covering family in $\C$, so $\W(\C,\D)$ is the maximal
  subgroupoid of $\W(\C\bs \D)$.  In this case, $I$ is not a homotopy
  equivalence.
\end{example}



Recall that $P = N\W(p)$.  By considering $Y_\dot$, $N\W(\C,\D)$ and
$N\W(\C\bs\D)$ as bisimplicial sets which are constant in the vertical
direction, we can factor $P$ as
\[N \W((\C/i)_\dot) \rto^{p_1} Y_\dot \rto^{\kappa} N \W(\C,\D)
\rto^{NI} N \W(\C\bs\D),\] where $p_1$ is defined on a simplex in
$N_m\W((\C/i)_0)$ by
\begin{align*}
  \makeshort{p_1(A_0 \rto\cdots\rto A_m)} &= \makeshort{(A_0 \rto \cdots \rto
    A_m)_{\C\bs\D}}, 
\end{align*}
and defined on simplices in $N_m\W((\C/i)_n)$ by $p_1\circ d_0^n$.  Here, the
use of $\kappa$ is a slight abuse of notation, as $\kappa$ is a functor and not
a morphism of simplicial sets; however, as the map we want applies $\kappa$ to
each $n$-simplex, we use it for clarity.  Thus it suffices to show that $p_1$,
$\kappa$ and $I$ are weak equivalences after geometric realization; this is
shown in Sections~\ref{sec:p_1},~\ref{sec:p_2} and~\ref{sec:iota}, respectively.

Theorem~\ref{thm:cofibercalc} has two conditions: that $\D$ is a sieve in $\C$
and that $\C$ has all complements for all objects in $\D$.  It turns out that
the second condition is only used to prove that $NI$ is a homotopy
equivalence, so we see that the map
\[|\W((\C/i)_\dot)| \rto |\W(\C,\D)|\]
is always a homotopy equivalence when $\D$ is a sieve in $\C$.  

\subsection{The map $p_1$} \label{sec:p_1}

First we consider $p_1$.  We begin with some technical results.

\begin{lemma} \label{lem:splituniondiag}
  Write 
  \begin{align*}
    C &= (A_0 \rto \cdots \rto A_n) \in N_n\W(\C), \\
    D &= (B_0 \rto \cdots \rto B_n) \in N_n\W(\D),
  \end{align*}
  and assume that for all $i$, the indexing sets of $A_i$ and $B_i$ are
  disjoint.  We write $C\cup D$ for the diagram 
  \[A_0\cup B_0 \rto A_1\cup B_1\rto\cdots \rto A_n \cup B_n,\] where we use
  $\cup$ instead of $\sqcup$ to emphasize that we only need to take simple
  unions of the indexing sets to form the coproduct.  Then we have
  \[C = C_{\C\bs\D} \cup C_\D \qquad (C\cup D)_{\C\bs\D} = C_{\C\bs\D} \qquad
  (C\cup D)_\D = C_\D\cup D.\]
\end{lemma}

\begin{proof}
  This follows directly from the definitions.  The reason we need the assumption
  on indexing sets is that if the indexing sets are not disjoint then we may end
  up taking different versions of the disjoint union on the two sides of the
  third equality.
\end{proof}

Since $\W(\C\vee\D^{\vee k})$ is a full subcategory of $\W(\C) \times \W(\D)^k$
consisting of those tuples with disjoint indexing sets, we can consider a
diagram $A_0\rto\cdots\rto A_m$ in $\W(\C\vee \D^{\vee k})$ as a $k+1$-tuple of
diagrams $(C,D_1,\ldots,D_k)$ with $C$ in $N_m\W(\C)$ and $D_i$ in $N_mW(\D)$ for all
$i$.  Note that, by definition, for each $j=0,\ldots,k$ the $j$-th objects in
all diagrams have disjoint indexing sets.  In the subsequent proofs we will often be
using this identification.

\begin{lemma} \label{lem:well-defp1}
  The map $p_1:N\W((\C/i)_\dot) \rto Y_\dot$ is a map of bisimplicial sets.
\end{lemma}

\begin{proof}
  The fact that $p_1$ commutes with the horizontal simplicial structure maps
  follows directly from the definitions.
  %
  It remains to check the vertical simplicial direction.  As all vertical
  simplicial structure maps on $Y_\dot$ are trivial, we just need to check that
  applying vertical simplicial structure maps does not affect the image under
  $p_1$.
  Consider $A_0\rto\cdots\rto A_m$ in $N_m\W(\C\vee \D^{\vee k})$ as a $k+1$-tuple
  of diagrams $(C,D_1,\ldots,D_k)$ By Lemma~\ref{lem:splituniondiag}
  $(C,D_1,\ldots,D_k)_{\C\bs\D} = C_{\C\bs\D}$, so we can conclude that $p_1$
  commutes with all of the vertical simplicial structure maps other than $d_0$.
  But (again by Lemma~\ref{lem:splituniondiag})
  \begin{align*}
    p_1d_0(C,D_1,\ldots,D_k) &= p_1(C\cup D_1,\ldots,D_k) = (C\cup D_1)_{\C\bs\D}
      \\ &= C_{\C\bs\D} = p_1(C,D_1,\ldots,D_k),
  \end{align*}
  so $p_1$ commutes with $d_0$ as well.
\end{proof}

We are now ready to prove that $p_1$ is an equivalence after geometric
realization.  We do this by showing that it is a levelwise homotopy
equivalence. 

\begin{lemma} \label{lem:discrete-levelwise} For all $n\geq 0$, the simplicial set
  $X_\dot = N_n \W((\C/i)_\dot)$ is homotopically discrete, and $\pi_0X_\dot =
  Y_n$.  When restricted to $X_\dot$ we have $p_1 = \pi_0$.
\end{lemma}

\begin{proof}
  Write $X_{-1} = Y_n$ and consider $p_1:X_0 \rto Y_n$ as an augmentation of
  $X$.  An augmented simplicial set with an extra degeneracy contracts onto its
  augmentation \cite[Lemma III.5.1]{g+j}, so it remains for us only to
  construct $s_{-1}: X_m \rto X_{m+1}$ for $m\geq -1$.
  
  We define the map $s_{-1}:Y_n \rto X_0$ to be the inclusion.  For $n\geq 0$,
  we define the map $s_{-1}:X_m \rto X_{m+1}$ in the following manner.  Consider
  a diagram $D = (A_0 \rto \cdots \rto A_n)$ in $\W((\C/i)_m)$.  Write this
  diagram as an $m+1$-tuple of diagrams $(C,D_1,\ldots,D_m)$ and define
  \[s_{-1}(D) = (C_{\C\bs\D}, C_\D, D_1,\ldots,D_m).\] Checking that $s_{-1}$
  satisfies the conditions to be an extra degeneracy and that $p_1$ is an
  augmentation is a straightforward application of
  Lemma~\ref{lem:splituniondiag}.
\end{proof}

\subsection{The map $\kappa$} \label{sec:p_2}

We now move on to $\kappa$.  
We begin with
several technical results.  When the proofs are straightforward from definitions
we omit them.  Recall that we write $\kappa = \W(c)$.

\begin{lemma} \label{lem:projpush} Let $\C$ be a closed assembler.  Suppose that
  $f:A \rto C$ and $g:B \rto C$ in $\W(\C)$ are two morphisms such that there
  exists $h:\kappa(A) \rto \kappa(B)$ in $\W(\C\bs\D)$ making the diagram
  \begin{diagram}
    {\kappa(A) & & \kappa(C) \\
      & \kappa(B)\\};
    \to{1-1}{1-3}^{\kappa(f)} \to{1-1}{2-2}+{outer sep=-.1ex,auto,swap}{h}
    \to{2-2}{1-3}+{outer sep=-.1ex,auto,swap}{\kappa(g)}
  \end{diagram}
  commute.  Then for any pullback square
  \begin{diagram}
    { D & B \\ A & C\\};
    \arrowsquare{g'}{f'}{f}{g}
  \end{diagram}
  we have $\kappa(D) = A$, $\kappa(f') = 1_{\kappa(A)}$ and $\kappa(g') = h$.
\end{lemma}


\begin{lemma} \label{lem:preim-id} If $f:A \rto B$ in $\W(\C)$ has $B_\D = \{\}$ and
  $\kappa(f) = 1_B$ then $f = 1_B$.
\end{lemma}


\begin{lemma} \label{lem:reducing}
  Suppose that we have a square
  \begin{diagram}
    {\SCob{A}{i} & \SCob{B}{j} \\ \SCob{C}{k} & \SCob{D}{l} \\};
	\arrowsquare{f}{f'}{g}{g'}
  \end{diagram}
  in $\W(\C)$ such that $\kappa(g) = \kappa(g')$.  Then $f_{\C\bs\D} =
  f'_{\C\bs\D}$.
\end{lemma}

\begin{proof}
  Let $\tilde J = \{j\in J\,|\, B_j\in \C\bs\D\}$, and define $\tilde K$
  analogously.  As $\kappa(g) = \kappa(g')$ it follows that $\tilde J = \tilde
  K$.  We first show that $f^{-1}(\tilde J) = (f')^{-1}(\tilde K)$.  Suppose
  that $i\in I$ is such that $f(i) \in \tilde J$.  Then we have the following
  square in $\C$:
  \begin{diagram}
    {A_i & B_{f(i)} \\ C_{f'(i)} & D_{gf(i)} \\};
	\arrowsquare{f_i}{f'_i}{g_{f(i)}}{g'_{f'(i)}}
  \end{diagram}
  Thus the morphisms $g_{f(i)}$ and $g'_{f'(i)}$ are not disjoint; since both
  appear in $\kappa(g)$ we must have $f(i) = f'(i)$, and thus $f'(i)$ is in
  $\tilde K$.  Thus $f^{-1}(\tilde J) \subseteq (f')^{-1}(\tilde K)$, and by
  symmetry these sets are equal.  But we also know that
  \[g_{f(i)} = g'_{f(i)} = g'_{f'(i)};\]
  since all morphisms in $\C$ are monic, we can conclude that $f_i = f'_i$.  
  Thus $f$ and $f'$ are equal when restricted to $\{A_i\}_{i\in f^{-1}(\tilde
    J)}$, which implies that $f_{\C\bs\D} = f'_{\C\bs\D}$.
\end{proof}

\begin{lemma} \label{lem:His0} Let $X_\dot$ be a simplicial set.  Suppose that
  for every $m$-cycle $\epsilon = \sum_{i=1}^k c_ix_i$ with $x_i \in X_m$ there
  exist simplices $\tilde x_i \in X_{m+1}$ such that $d_0\tilde x_i = x_i$ and
  whenever $d_jx_i = d_{j'}x_{i'}$ it is also the case that $d_{j+1}\tilde x_i =
  d_{j'+1}\tilde x_{i'}$. Then $H_mX_\dot = 0$.
\end{lemma}

\begin{proof}
  Observe that for all $i$ and $j$, the terms $d_jx_i$ are elements of a basis
  for a free abelian group, so that the only way that
  $\partial \epsilon = \sum_{i=1}^k c_i \sum_{j=0}^m (-1)^j d_j x_i$ can be zero
  is if the coefficients of each basis element are $0$.  Therefore the $d_jx_i$
  match up in some way which ensures that the coefficients of the simplices in
  the end are zero.

  Let $\beta = \sum_{i=1}^k c_i \tilde x_i$.  Then
  \begin{align*}
    \partial \beta &=
                     \partial \sum_{i=1}^k c_i\tilde x_i  = \sum_{i=1}^k
                     \bigg(\sum_{j=1}^{m+1} (-1)^jc_id_j\tilde x_i +
                     c_id_0\tilde x_i\bigg) \\
                   &= \sum_{i=1}^k c_i\sum_{j=1}^{m+1}(-1)^jd_j\tilde x_i+
                     \epsilon = \epsilon.
  \end{align*}
  The last step follows because the construction of $\tilde x_i$ ensures that
  the terms in the sum cancel the same way that they did for $\epsilon$.
\end{proof}

To show that $\kappa$ is a weak equivalence, we use the simplicial version of
Quillen's Theorem A (see \cite[Proposition 1.4.A]{waldhausen83}), which 
states that the following lemma implies that $\kappa$ is a homotopy equivalence.

\begin{lemma} \label{lem:simplyconnected} Let $\alpha:\Delta^n \rto N
  \W(\C\bs\D)$ be any $n$-simplex of $ N \W(\C,\D)$.  Define
  $\kappa/(n,\alpha)$ to be the pullback of the diagram
  \[\Delta^n \rto^\alpha  N \W(\C,\D)  \lto^{\kappa}Y_\dot .\]
  Then $\kappa/(n,\alpha)$ is contractible for all $n$ and $\alpha$.
\end{lemma}

\begin{proof}
  We prove this by showing three things:
  \begin{itemize}
  \item[(a)] $\kappa/(n,\alpha)$ is connected,
  \item[(b)] $\pi_1(\kappa/(n,\alpha)) = 0$, and
  \item[(c)] for all $i>0$, $H_i(\kappa/(n,\alpha)) = 0$.  
  \end{itemize}
  Together, these imply that all homotopy groups of $\kappa/(n,\alpha)$ are trivial,
  showing that it is contractible.  Write 
  \[\alpha = A_0\rto A_1 \rto \cdots \rto A_n \in N_n\W(\C,\D),\]
  and consider $\Delta^n$ to be the nerve of $0\rto 1\rto\cdots\rto n$.  Fix a
  lift
  \[\tilde \alpha = A_0\sqcup D_0 \rto A_1 \sqcup D_1 \rto \cdots
  \rto A_n \in Y_n\] of $\alpha$ along $\kappa$, where $D_i$ is in $\W(\D)$; such a
  lift exists by the definition of $\W(\C,\D)$.  Note that there is no $D_n$ at
  the end of the diagram, as all elements of $Y_n$ have last term in
  $\W(\C\bs\D)$. We prove (a)-(c) in order.

  \noindent
  Proof of (a): The vertices of $\kappa/(n,\alpha)$ are pairs $(i,A)$ for $0\leq i
  \leq n$ and $A$ in $Y_0$ such that $\alpha(i) = \kappa(A)$.  By definition, the
  elements of $Y_0$ are exactly objects in $\W(\C\bs\D)$, and $\kappa$ is the
  identity of these.  Thus the vertices of $\kappa/(n,\alpha)$ are pairs $(i,A_i)$.
  The vertices of $\tilde\alpha$ are $A_0,A_1,\ldots,A_n$.  Thus there exists a
  simplex of $\kappa/(n,\alpha)$ which contains all of its vertices; since a
  simplex is connected $\kappa/(n,\alpha)$ is, as well.

  \noindent
  Proof of (b): An element of $\pi_1(\kappa/(n,\alpha))$ is represented by a zigzag
  \[\makeshort[.6]{i_0 \rto^{\beta_0} i_1 \lto^{\beta_1} i_2 \rto^{\beta_2} \cdots
    \rto^{\beta_m} i_m \in \Delta^n,}\] with $i_0 = i_m$, together with morphisms
  $f_j: A_{i_j}\sqcup E_j \rto A_{i_{j'}}$ in $Y_1$, where $\beta_j:i_j \rto
  i_{j'}$ and $E_j$ is in $\W(\D)$.  We show that any such zigzag is
  null-homotopic.  In this, we use the usual approach for replacing zigzags in
  categories closed under pullbacks: we show that we can replace any two arrows
  going in the same direction by a ``composition'', and then we show that we can
  replace arrows with the same codomain by arrows with the same domain using the
  pullback.  These two steps allow us to reduce the problem to zigzags of length
  $1$ or $2$.  Finally, we analyze these two cases.  We include all details of
  this approach because the unusual simplicial structure on $Y_\dot$ introduces
  several subtleties.

  First, suppose that two successive $\beta_j$'s go in the same direction;
  without loss of generality, we assume that $\beta_j:i_j\rto i_{j+1}$ and
  $\beta_{j+1}:i_{j+1} \rto i_{j+2}$.  Then we can construct a $2$-simplex
  represented by
  \[\ms[4em]{A_{i_j} \sqcup E_j \sqcup E_{j+1} \rto^{f_j\sqcup 1_{E_{j+1}}}
    A_{i_{j+1}}\sqcup E_{j+1} \rto^{f_{j+1}} A_{i_{j+2}}},\] which shows that
  this zigzag is homotopic to the one where $\beta_j$ and $\beta_{j+1}$ are
  replaced by $i_j \rto i_{j+2}$ with morphism $A_{i_j}\sqcup (E_j \sqcup
  E_{j+1}) \rto A_{i_{j+2}}$.

  Now suppose that $\beta_j:i_j \rto i_{j+1}$ and $\beta_{j+1}: i_{j+2} \rto
  i_{j+1}$ and assume without loss of generality that $i_j \leq i_{j+2}$.  Take
  the pullback of $f_j$ and $f_{j+1}$ to form a square
  \begin{diagram}
    { A_{i_j} \sqcup E & A_{i_{j+2}}\sqcup  E_{j+2} \\
      A_{i_j} \sqcup E_j & A_{i_{j+1}} \\};
    \arrowsquare{f_j'}{f_{j+1}'}{f_{j+1}}{f_j}
  \end{diagram}
  The pullback is of the form $A_{i_j} \sqcup E$ by Lemma~\ref{lem:projpush} and
  because $i_j \leq i_{j+2}$.  Thus we can replace these two simplices in the
  zigzag by the two simplices represented by
  \[(i_j \rto i_j, (A_{i_j} \sqcup E \rto A_{i_j} \sqcup E_j)_{\C\bs\D})\] and
  \[(i_j \rto i_{j+1}, (A_{i_j}\sqcup E \rto A_{i_{j+1}}\sqcup
  E_{j+1})_{\C\bs\D}),\] thus reversing the directions of the two arrows in the
  zigzag.  (By Lemma~\ref{lem:preim-id} the first of these is trivial, but that
  does not matter for the current analysis.)

  We now only need to consider zigzags of length $1$ or $2$.  A zigzag of length
  $1$ is a morphism $A_i \sqcup E \rto A_i$ which projects down to the identity
  morphism in $\W(\C\bs\D)$.  By Lemma~\ref{lem:preim-id}, this morphism must be
  the identity morphism, and thus such a zigzag is represented by a degenerate
  simplex, and is therefore contractible.  Now consider a zigzag of length $2$.
  Such a zigzag is represented by a pair of morphisms
  \[A_i \lto^{f_1} A_j \sqcup E \qquad A_j \sqcup E' \rto^{f_2} A_i.\] We thus
  have a pullback square
  \begin{diagram}
    { A_j \sqcup E'' & A_j \sqcup E' \\ A_j \sqcup E & A_i \\};
    \arrowsquare{f_2'}{f_1'}{f_2}{f_1} 
    \arrow{->,densely dotted}{1-1}{2-2}^h
  \end{diagram}
  where the pullback is $A_j\sqcup E''$ and $(f_1')_{\C\bs\D} = (f_2')_{\C\bs\D}
  = 1_{A_j}$ by Lemma~\ref{lem:projpush}.  We can replace this zigzag by
  \[\makeshort{(j \rto i, h), (j \rto j, (f_1')_{\C\bs\D}), (j \rto
    j, (f_2')_{\C\bs\D}), (j \rto i, h)}.\]
  As $(f_1')_{\C\bs\D} = (f_2')_{\C\bs\D} = 1_{A_j}$, the middle two simplices
  are degenerate and can be removed.  The outside two simplices are just the
  same simplex repeated twice, so we see that this zigzag is also
  null-homotopic.

  \noindent
  Proof of (c): Let $m\geq 2$, and let $\epsilon = \sum_{i=1}^k c_i(\alpha_i,x_i)$
    be an $m$-cycle in $\kappa/(n,\alpha)$.  We use Lemma~\ref{lem:His0} to
    prove that $\epsilon = \partial \beta$ for some $\beta$.

    We write $x_i\in Y_m$ as
    \[X_{i0} \rto X_{i1} \rto \cdots \rto X_{im},\] where $X_{im} =
    A_{\alpha_i(m)}$.  For each simplex $(\alpha_i,x_i)$ we have a diagram
    \[\ms{X_{i0} \sqcup D_{\alpha_i(m)} \rto \cdots\rto X_{im}\sqcup D_{\alpha_i(m)} \rto A_n \in \W(\C)}.\] Let $\I$ be the
    category with
    \[\ob \I = \{1,\ldots,k\}\times \{0,\ldots,m+1\} / ((i,m+1)\sim (i',m+1))\]
    for all $1\leq i,i'\leq k$, and
    \[\I((i,j),(i',j')) = \begin{cases} * \caseif i = i'\hbox{ and } j \leq
      j', \\
      * \caseif j' = m+1 \\ \emptyset \caseotherwise.\end{cases}\] Let $\chi: \I
    \rto \W(\C)$ be the functor taking the pair $(i,j)$ to $X_{ij}\sqcup
    D_{\alpha_i(m)}$ and the morphisms to the morphisms defined by the simplices
    above.  We define $\chi(i,m+1) = A_n$.  Let $\tilde X$ be the limit of this
    diagram, which exists because $\W(\C)$ has pullbacks, and write $f_i: \tilde
    X \rto X_{i0}$.  Let $\mu = \min_i \alpha_i(0)$; note that by repeated
    application of Lemma~\ref{lem:projpush} we have $\tilde X_{\C\bs\D} =
    A_\mu$.  Let $\tilde \alpha_i: [m+1] \rto [n]$ be defined by $\tilde
    \alpha_i(j) = \alpha_i(j-1)$ for $j>0$, and $\tilde \alpha_i(0) = \mu$, and
    let $\tilde x_i$ be the simplex represented by the diagram
    \[\big(\tilde X \rto^{f_i} X_{i0}\sqcup D_{\alpha_i(m)}\rto X_{i1} \sqcup D_{\alpha_i(m)}\rto \cdots
    \rto X_{im} \sqcup D_{\alpha_i(m)}\big)_{\C\bs\D},\] which we write 
    \[\tilde X_i \rto^{h_i} X_{i0} \rto X_{i1} \rto\cdots\rto X_{im}.\] 
    By definition, $d_0(\tilde \alpha_i, \tilde x_i) = (\alpha_i,x_i)$, so to
    apply Lemma~\ref{lem:His0} it remains to check that if $d_j(\alpha_i,x_i) =
    d_{j'}(\alpha_{i'}, x_{i'})$ then $d_{j+1}(\tilde \alpha_i, \tilde x_i) =
    d_{j'+1}(\tilde \alpha_{i'}, \tilde x_{i'})$..

    Let $0\leq j,j' \leq m$.  We want to show that if $d_j(\alpha_i,x_i) =
    d_{j'}(\alpha_{i'},x_{i'})$ then $d_{j+1}(\tilde \alpha_i, \tilde x_i) =
    d_{j'+1}(\tilde\alpha_{i'}, \tilde x_{i'})$
    We need to show that the two diagrams
    \begin{align*}
    D &= \makeshort{(\tilde X \rto X_{i0} \sqcup D_{\alpha_i(m)} \rto \cdots} \\
    & \qquad\ms{\cdots\rto \widehat
      X_{ij} \sqcup D_{\alpha_i(m)} \rto \cdots \rto X_{im}\sqcup
      D_{\alpha_i(m)})}_{\C\bs\D}
    \end{align*}
    and
    \begin{align*}
      D' &= \makeshort{(\tilde X \rto X_{i'0} \sqcup D_{\alpha_{i'}(m)} \rto\cdots} \\
      &\qquad\ms{\cdots \rto \widehat X_{i'j'} \sqcup D_{\alpha_{i'}(m)} \rto \cdots \rto
      X_{i'm}\sqcup D_{\alpha_{i'}(m)})}_{\C\bs\D}
    \end{align*}
    are equal.
    Note that 
    \[\makeshort{D = Y \rto X_{i0} \rto\cdots \rto \widehat X_{ij} \rto\cdots
      \rto X_{im}}\]
    for some $Y$, and analogously for $D'$.  The assumption that $d_jx_i =
    d_{j'}x_{i'}$ implies that the composition of all but the first morphisms
    are the same in both diagrams, so (as all morphisms in $\W(\C)$ are monic by
    Proposition~\ref{prop:Wprop}(1)) it suffices to show that the total composition of $D$
    is the same as the total composition of $D'$.

    For the following discussion, we assume that $0 < j,j' < m$; however,
    as this is assumed only for ease of notation, it does not affect the proof.
    We have the following diagram:
    \begin{general-diagram}{.9em}{3em}
      {& X_{i0}\sqcup D_{\alpha_i(m)} & X_{im} \sqcup D_{\alpha_i(m)} \\
        \tilde X &&&  A_m \\
        & X_{i'0} \sqcup D_{\alpha_{i'}(m)} & X_{i'm} \sqcup D_{\alpha_{i'}(m)}
        \\};
      \to{1-2}{1-3}^{f\sqcup 1} \to{3-2}{3-3}^{f'\sqcup 1}
      \to{2-1}{1-2}^{f_i} \to{2-1}{3-2}_{f_{i'}}
      \to{1-3}{2-4}^g \to{3-3}{2-4}_{g'}
    \end{general-diagram}
    As $d_j\alpha_i = d_{j'}\alpha_{i'}$ we know that $\alpha_i(m) =
    \alpha_{i'}(m)$, and in particular $\kappa(g) = \kappa(g')$.  Thus
    Lemma~\ref{lem:reducing} applies, and
    \[h_if = (f_i(f\sqcup 1))_{\C\bs\D} = (f'_{i'}(f'\sqcup 1)))_{\C\bs\D} =
    h_{i'}f'.\] (When $j$ and $j'$ are either $0$ or $m$ this just affects the
    indices in the above diagram; we define $f$ and $f'$ to be the full
    compositions in $d_jx_i$ and $d_{j'}x_{i'}$, respectively.)
\end{proof}

\subsection{The functor $I$} \label{sec:iota}

Lastly we consider $I$.  First, we need the following technical result,
which allows us to remove ``errors'' in extending weak equivalences in
$\W(\C\bs\D)$ to weak equivalences in $\W(\C)$.

\begin{lemma} \label{lem:cutting-up} Suppose that $\C$ is a closed assembler and
  has complements for all objects in $\D$.  Suppose we are given a finite
  collection of morphisms $f_i:A_i \rto B$ in $\C$ such that for $i\neq i'$,
  $A_i\times_B A_{i'}$ is in $\D$.  Then there exists a morphism
  \[g:\SCob{Z}{k} \rto \SCob{A}{i}\] in $\W(\C)$ such
  that for all $k\neq k'\in K$ the morphisms $f_{g(k)}g_k$ and $f_{g(k')}g_{k'}$
  are either equal or disjoint.  If they are equal then $g(k) \neq g(k')$.
\end{lemma}

\begin{proof}
  We prove this by induction on $|I|$.  If $|I| \leq 1$ then the lemma holds
  tautologically, so we only need to check the inductive step.

  Pick $i_0\in I$, and consider the collection $\{f_i\}_{i\neq i_0}$.  The
  conditions of the lemma apply to this collection, and thus there exists
  $g:\{Z_k\}_{k\in K} \rto \{A_i\}_{i\neq i_0}$ satisfying the conditions of the
  lemma.  As $\C$ has complements for all objects of $\D$ and $\D$ is a sieve in
  $\C$, $\C$ has complements for $Z_k\times_B A_{i_0}$ (which sits above
  $A_{g(k)}\times_B A_{i_0}$) for all $k\in K$.  Let $\mathscr{G}_k$ be the
  finite disjoint covering family of $Z_k$ which includes the morphism
  $Z_k\times_B A_{i_0} \rto Z_k$; let $\mathscr{F}_k$ be the finite disjoint
  covering family of $A_{i_0}$ which includes this morphism.  Now let
  $\mathscr{F}$ be a common refinement of the $\mathscr{F}_k$'s and let
  $\mathscr{G}_k'$ be the refinement of $\mathscr{G}_k$ by $\mathscr{F}$.  By
  definition all of these families are finite disjoint covering families.

  A morphism $\{W_l\}_{l\in L}\rto \{A_i\}_{i\in I}$ is a collection of finite
  disjoint covering families, one for each $i\in I$.  Thus in order to construct
  the desired morphism, it suffices to give a covering family for each $i\in I$.
  For $i_0$, we take $\mathscr{F}$.  For $i\neq i_0$, take a refinement of
  $\{g_k:Z_k \rto A_i\}_{k\in g^{-1}(i)}$ by the relevant $\mathscr{G}'_k$.
  Then this is the desired morphism.

  The last condition in the lemma follows because if $g(k) = g(k')$ then (as all
  morphisms are monic) $g_k = g_{k'}$ and the covering family 
  \[\{g_l:Z_l \rto A_{g(k)}\}_{l\in g^{-1}(g(k))}\] is not a finite disjoint
  covering family; hence, $g$ is not a valid morphism of $\W(\C)$, which is a
  contradiction.  Thus $g(k)\neq g(k')$ if the two morphisms $f_{g(k)}g_k$ and
  $f_{g(k')}g_{k'}$ are equal.
\end{proof}

In order to prove that $|I|$ is a homotopy equivalence we need a special
case of Quillen's Theorem A.

\begin{lemma} \label{lem:specialcase} Let $\C$ be a category where all morphisms
  are monomorphisms, and where any diagram $B\rto A \lto C$ can be completed to
  a commutative square.  Suppose that $\D$ is a subcategory which contains all
  objects and has the property that for any morphism $f:A \rto B$ in $\C$ there
  exists a morphism $g:Z\rto A$ in $\D$ such that $fg$ is in $\D$.  Then the
  inclusion $F:\D\rto \C$ is a homotopy equivalence.
\end{lemma}

\begin{proof}
  We use Quillen's Theorem A to show that this functor is a homotopy
  equivalence.  As all morphisms are monomorphisms in $\C$, $\C/A$ is a preorder
  for all $A$.  Thus $F/A$ is also a preorder, and by Lemma~\ref{lem:cofiltpre}
  it suffices to show that it is cofiltered to show that it is contractible.
  Suppose that we are given two objects $f:B \rto A$ and $g:C\rto A$ in $F/A$.
  We then have the following diagram:
  \begin{diagram}
    { Z & Y & X & C \\ && B & A \\};
    \to{1-3}{1-4}^{f'} \to{1-3}{2-3}_{g'} \to{2-3}{2-4}^f \to{1-4}{2-4}^g
    \to{1-1}{1-2}^\beta \to{1-2}{1-3}^\alpha
  \end{diagram}
  Here, $f'$ and $g'$ complete $f$ and $g$ to a commutative square, as we
  assumed was possible in $\C$.  Let $\alpha$ be a morphism $Y \rto X$ in $\D$
  such that $g'\alpha$ is in $\D$; there exists such an $\alpha$ by the property
  in the lemma.  Similarly, let $\beta$ be a morphism $Z \rto Y$ in $\D$ such
  that $f'\alpha\beta$ is in $\D$.  Thus the morphism $Z \rto C$ is in $\D$, and
  the morphism $Z \rto B$ is in $\D$, so the object $gf'\alpha\beta:Z \rto A$ is
  an object of $F/A$ over both $f$ and $g$.
\end{proof}

We are now ready to show that $|I|$ is a homotopy equivalence.

\begin{lemma} \label{lem:iotagood}
  $|I|$ is a homotopy equivalence.
\end{lemma}

\begin{proof}
  We show that $|I|$ is a homotopy equivalence using
  Lemma~\ref{lem:specialcase}.  In $\W(\C\bs\D)$ all morphisms are monic and any
  diagram $B\rto A \lto C$ can be completed to a commutative square.  The only
  property that remains to be checked is that for any morphism $f:A\rto B$ in
  $\W(\C\bs\D)$, there exists a morphism $g:Z \rto A$ in $\W(\C,\D)$ such that
  $fg$ is in $\W(\C,\D)$.  As any morphism in $\W(\C\bs\D)$ can be written as a
  coproduct of morphisms $f:\SCob{A}{i} \rto \{B\}$ it suffices to check that
  $g$ exists for such morphisms.

  Consider any morphism $\tilde f:\{A_i\}_{i\in \tilde I}\rto \{B\}\in
  \W(\C\bs\D)$. The family $\{\tilde f_i:A_i \rto B\}_{i\in \tilde
    I}$ can be completed to a covering family in $\C$ by morphisms with domains
  in $\D$; write this whole family $\{f_i:A_i \rto B\}_{i\in I}$.  It satisfies
  the condition of Lemma~\ref{lem:cutting-up}, and thus we have a 
  \[\hat g:Z=\SCob{Z}{k} \rto \{A_i\}_{i\in I}\in \W(\C)\] 
  such that for all $k\neq k'\in K$, $f_{g(k)}g_k$ and $f_{g(k')}g_{k'}$ are
  either disjoint or equal.  Thus the family $\mathcal{F}=\{f_{g(k)}g_k:Z_k \rto
  B\}_{k\in K}$ is a finite covering family, but it may not be disjoint because
  for some $k\neq k'$ $f_{g(k)}g_k$ might equal $f_{g(k')}g_{k'}$.  However, in
  this case, $Z_k$ can complete $f_{g(k)}$ and $f_{g(k')}$ to a square.  By
  Lemma~\ref{lem:cutting-up} in this case $g(k) \neq g(k')$, which means that
  $Z_k$ is in $\D$.  Thus the only obstructions to $\mathcal{F}$ defining a morphism
  of $\W(\C)$ is duplication of some morphisms with domains in $\D$.
  Consequently, the morphism defined by the family
  \[\{f_{g(k)}g_k:Z_k \rto B\}_{k\in \tilde K},\] where $\tilde K = \{k\in
  K\,|\, Z_k\in \C\bs\D\}$, is a valid morphism of $\W(\C,\D)$.  We therefore
  define $g:\{Z_k\}_{k\in \tilde K} \rto \{A_i\}_{i\in \tilde I}$ to be the
  restriction of $\hat g$ to $\tilde K$; this also gives a valid morphism of
  $\W(\C,\D)$.

  To check that $g$ satisfies the conditions of Lemma~\ref{lem:specialcase} it
  remains to check that the composition $\tilde fg$ is in $\W(\C,\D)$.  $\tilde
  fg$ is the morphism $\{Z_k\}_{k\in \tilde K} \rto \{B\}$; we need to check
  that this covering family can be completed to a finite disjoint covering
  family in $\C$.  Consider the set $\{f_{\hat g(k)}\hat g_k\,|\, k\in K\}$.
  These are a covering family of $B$ in $\C$, as they are the refinement of a
  covering family $\{f_i:A_i \rto B\,|\, i\in I\}$ by the covering families
  $\{Z_k \rto A_i\,|\, k\in \hat g^{-1}(i)\}$.  From the construction of $\hat
  g$, all distinct elements in the set are disjoint, so this is a finite
  disjoint covering family which is the completion of the family we are
  considering.
\end{proof}

This wraps up the proof of Theorem~\ref{thm:cofibercalc}.  
\bibliography{IZ-all}

\end{document}